\documentclass[12pt,a4paper]{article}
\usepackage{latexsym,amssymb,amsfonts,amsmath,amsthm,nccmath,float,enumitem,setspace,bm,authblk,longtable}
\usepackage[usenames,dvipsnames]{xcolor}
\usepackage[hidelinks]{hyperref}
\usepackage[margin=2cm]{geometry}
\usepackage{booktabs}
\usepackage[section]{placeins}
\usepackage{amssymb}
\usepackage{lipsum}
\usepackage{tabularx}
\usepackage{graphicx}
\usepackage{float}
\usepackage{makecell}
\setlist{topsep=3pt,partopsep=0pt,itemsep=1pt,parsep=0pt}

\hypersetup{
    colorlinks,
    linkcolor={red!80!black},
    citecolor={blue!80!black},
    urlcolor={blue!80!black}
}

\newtheorem{Theorem}{Theorem}[section]
\newtheorem{Proposition}{Proposition}

\newtheorem{Remark}{Remark}[section]
\newtheorem{Lemma}{Lemma}[section]

\newtheorem{Construction}{Construction}[section]

\newcommand{\G}{\mathcal{G}}
\newcommand{\Z}{\mathbb{Z}}

\newcommand{\B}{\mathcal{B}}
\newcommand{\F}{\mathcal{F}}

\def \dev {{\rm dev}}

\def \leq {\leqslant}
\def \geq {\geqslant}

\def \mod#1{{\:({\rm mod}\ #1)}}

\let\oldproofname=\proofname
\renewcommand{\proofname}{\rm\bf{\oldproofname}}

\makeatletter

\newcommand{\Rmnum}[1]{\expandafter\@slowromancap\romannumeral #1@}
\makeatother

\begin{document}

\title{Cyclic relative difference families with block size four and their applications}

\author[a]{Chenya Zhao}
\author[a]{Binwei Zhao}
\author[a]{Yanxun Chang \thanks{Supported by NSFC under Grant 11971053}}
\author[a]{Tao Feng \thanks{Supported by NSFC under Grant 12271023}}
\author[b]{Xiaomiao Wang \thanks{Supported by NSFC under Grant $11771227$}}
\author[a]{Menglong Zhang}
\affil[a]{School of Mathematics and Statistics, Beijing Jiaotong University, Beijing 100044, P.R. China}
\affil[b]{School of Mathematics and Statistics, Ningbo University, Ningbo 315211, P.R. China}
\renewcommand*{\Affilfont}{\small\it}
\renewcommand\Authands{ and }

\affil[ ]{cyazhao@bjtu.edu.cn; bwzhao@bjtu.edu.cn; yxchang@bjtu.edu.cn, tfeng@bjtu.edu.cn, wangxiaomiao@nbu.edu.cn, mlzhang@bjtu.edu.cn}
\date{}

\maketitle

\begin{abstract}
Given a subgroup $H$ of a group $(G,+)$, a $(G,H,k,1)$ difference family (DF) is a set $\F$ of $k$-subsets of $G$ such that $\{f-f':f,f'\in F, f\neq f',F\in \F\}=G\setminus H$. Let $g\Z_{gh}$ is the subgroup of order $h$ in $\mathbb Z_{gh}$ generated by $g$. A $(\Z_{gh},g\Z_{gh},k,1)$-DF is called cyclic and written as a $(gh,h,k,1)$-CDF. This paper shows that for $h\in\{2,3,6\}$, there exists a $(gh,h,4,1)$-CDF if and only if $gh\equiv h\pmod{12}$, $g\geq 4$ and $(g,h)\not\in\{(9,3),(5,6)\}$. As a corollary, it is shown that a 1-rotational S$(2,4,v)$ exists if and only if $v\equiv4\pmod{12}$ and $v\neq 28$. This solves the long-standing open problem on the existence of a 1-rotational S$(2,4,v)$. As another corollary, we establish the existence of an optimal $(v,4,1)$-optical orthogonal code with $\lfloor(v-1)/12\rfloor$ codewords for any positive integer $v\equiv 1,2,3,4,6\pmod{12}$ and $v\neq 25$. We also give applications of our results to cyclic group divisible designs with block size four and optimal cyclic $3$-ary constant-weight codes with weight four and minimum distance six.
\end{abstract}


\section{Introduction}

Throughout this paper, every union of sets will be understood as multiset union with multiplicities of elements preserved. Denote by $\Z_v$ the cyclic group of order $v$.

Let $H$ be a subgroup of a finite group $(G,+)$ and let $k$ be an integer satisfying $2\leq k\leq |G|$. A $(G,H,k,1)$-{\em relative difference family} (briefly DF) is a set $\F$ of $k$-subsets of $G$ (called {\em base blocks}) such that the list
$$\Delta \F:=\bigcup_{F\in \F}\Delta F:=\bigcup_{F\in \F}\{f-f':f,f'\in F, f\neq f'\}$$
covers each element of $G\setminus H$ exactly once and each element of $H$ zero time. There are $(|G|-|H|)/(k(k-1))$ base blocks in a $(G,H,k,1)$-DF and hence, a necessary condition for the existence of such a family is $|G|\equiv |H|\pmod{k(k-1)}$. Another obvious necessary condition is $|G|\geq k|H|$; otherwise, given any base block $F\in \F$, there exist two elements of $F$ that lie in a same right coset of $H$ in $G$, a contradiction.

\begin{Lemma}\label{lem:nece-DF}
If there exists a $(G,H,k,1)$-DF, then $|G|\equiv |H|\pmod{k(k-1)}$ and $|G|\geq k|H|$.
\end{Lemma}

When $H$ consists only of the identity $1_G$ of $G$, a $(G,1_G,k,1)$-DF is often written as a $(G,k,1)$-DF. When $G$ is a cyclic group, a $(G,H,k,1)$-DF is said to be \emph{cyclic} and written as a  $(|G|,|H|,k,1)$-CDF.

Difference families arose from the study of designs with a given group of automorphisms. We
here only briefly recall what has been done for cyclic difference families with base block size three or four. The reader is referred to \cite{ab,bjl} for more topics on difference families.

As far back as 1939, Peltesohn \cite{Peltesohn} showed that for $h\in\{1,3\}$, a $(gh,h,3,1)$-CDF exists if and only if $gh\equiv h \pmod{6}$ and $(g,h)\neq (3,3)$. This gives a solution to the Heffter's difference problems \cite{Heffter1896,Heffter1897} that yield the existence of cyclic Steiner triple systems (see Chapter VII.4 in \cite{bjl}). For general $h$, Wang and Chang \cite[Theorem 3.11]{wc} in 2009 proved that a $(gh,h,3,1)$-CDF exists if and only if (1) $gh\equiv h \pmod{6}$ and $g\geq 4$; (2) $g\not\equiv 2,3 \pmod{4}$ when $h\equiv 2 \pmod{4}$. This implies the existence of cyclic group divisible designs with block size three without short orbits (see \cite[Theorem 5.4]{wc}).

When $p\equiv 1\pmod{12}$ is a prime, Bose \cite{Bose} provided a sufficient condition for the existence of a $(p,4,1)$-CDF admitting a multiplier of order 3, and the necessary and sufficient condition for this special kind of CDFs was established by Buratti in \cite{Buratti95}. Improving Buratti's work in \cite{Buratti95-1}, Chen and Zhu \cite{chen45} showed that a $(p,4,1)$-CDF exists for any prime $p\equiv 1\pmod{12}$. Even though a $(v,4,1)$-CDF of composite order $v=v_1v_2$ can be obtained by means of recursive constructions in the hypothesis that a $(v_i,4,1)$-CDF exists (cf. \cite{Buratti98}) or by direct constructions with nice algebraic properties (cf. \cite{bp2}), no infinite family of $(v,4,1)$-CDFs such that $v$ can run over a congruent class was known until quite recently. Zhang, Feng and Wang \cite{Zhang1} recently developed an efficient direct construction method to establish the existence of $(gh,h,4,1)$-CDFs for $h\in\{1,4\}$, which solves the long-standing open problem on the existence of cyclic Steiner systems with block size four.

\begin{Theorem}\label{thm:cyclic design k=4} {\rm \cite{Zhang1}}
For $h\in\{1,4\}$, there exists a $(gh,h,4,1)$-CDF if and only if $gh\equiv h\pmod{12}$ and $(g,h)\not\in\{(25,1),(4,4),(7,4)\}$.
\end{Theorem}

This paper is a continuation of the study of \cite{Zhang1}. As the main result of this paper, we are to prove the following theorem.

\begin{Theorem}\label{thm:main}
For $h\in\{2,3,6\}$, there exists a $(gh,h,4,1)$-CDF if and only if $gh\equiv h\pmod{12}$, $g\geq 4$ and $(g,h)\not\in\{(9,3),(5,6)\}$.
\end{Theorem}

Section \ref{sec-2-3-6} gives a proof of Theorem \ref{thm:main}. In Section \ref{sec-recursions}, using a recursive construction, we construct more new infinite families of cyclic difference families (see Theorem \ref{thm:RDF-general}). Section \ref{sec-3} is devoted to providing applications of cyclic difference families to cyclic group divisible designs (see Theorem \ref{thm:CGDD-general}), 1-rotational Steiner systems (see Theorem \ref{thm:1-rotational}), optimal cyclic constant-weight codes (see Theorem \ref{thm:cyclic code}) and optimal optical orthogonal codes (see Theorem \ref{thm:OOC}). Concluding remarks are given in Section \ref{conclusion}.

\section{Existence of $(gh,h,4,1)$-cyclic difference families}

\subsection{The cases $h=2,3$ and $6$}\label{sec-2-3-6}

For positive integers $a,b$ and $c$ such that $a\leq b$ and $a\equiv b\pmod{c}$, we set $[a,b]_c:=\{a+ci:0\leq i\leq (b-a)/c\}$. When $c=1$, $[a,b]_1$ is simply written as $[a,b]$.

\begin{Lemma}\label{lem-RCDF-main}
For $h\in\{2,3,6\}$, there exists a $(gh,h,4,1)$-CDF for any $gh\equiv h\pmod{12}$ and $gh\geq 218$.
\end{Lemma}

\begin{proof}
Let $gh=72t+12x+h$, where $t\geq 3$ and $0\leq x\leq 5$. Then a $(gh,h,4,1)$-CDF contains $6t+x$ base blocks. We give a direct construction for such a CDF here. All base blocks are divided into two parts. The first part consists of $6t-18$ base blocks:
\begin{center}
\begin{tabular}{llll}
$F_{1,i}=\{0$, & $43t+a_1+i$, & $31t+a_2+2i$, & $8t+a_3+3i\}$,\\
$F_{2,i}=\{0$, & $23t+b_1+i$, & $5t+b_2+2i$, & $8t+b_3+3i\}$,\\
$F_{3,i}=\{0$, & $41t+c_1+i$, & $25t+c_2+2i$, & $8t+c_3+3i\}$,\\
$F_{4,i}=\{0$, & $35t+d_1+i$, & $5t+d_2+2i$, & $d_3+3i\}$,\\
$F_{5,i}=\{0$, & $47t+e_1+i$, & $19t+e_2+2i$, & $e_3+3i\}$,\\
$F_{6,i}=\{0$, & $21t+f_1+i$, & $13t+f_2+2i$, & $f_3+3i\}$,\\
\end{tabular}
\end{center}
where $1\leq i\leq t-2$ and $i\neq \lfloor t/2\rfloor$, and $a_j,b_j,c_j,d_j,e_j,f_j$ are shown in Table \ref{tab:coffi}. The remaining $18+x$ base blocks are listed in Appendix \ref{Sec:A} according to the parity of $t$. Write $\F_r=\{F_{r,i}: 1\leq i\leq{t-2}, i\neq \lfloor t/2\rfloor\}$ for $1\leq r\leq6$, and denote by $\F^{'}$ the set consisting of the remaining $18+x$ base blocks. It remains to check that $\F=(\bigcup_{r=1}^6{\F_r})\cup{\F^{'}}$ forms a $(gh,h,4,1)$-CDF.
\begin{table}
\begin{tabular}{|l|l|lll|lll|lll|lll|lll|lll|}\hline
$h$ & $x$ & $a_1$ & $a_2$ & $a_3$ & $b_1$ & $b_2$ & $b_3$ & $c_1$ & $c_2$ & $c_3$ & $d_1$ & $d_2$ & $d_3$ & $e_1$ & $e_2$ & $e_3$ & $f_1$ & $f_2$ & $f_3$\\\hline
& $0$ & $1$ & $1$ & $2$ & $1$ & $0$ & $1$ & $0$ & $1$ & $0$ & $1$ & $1$ & $1$ & $2$ & $2$ & $2$ & $1$ & $1$ & $0$ \\
& $1$ & $8$  & $7$  & $5$  & $4$ & $3$  & $4$ & $7$  & $4$  & $3$  & $7$  & $2$ & $1$ & $9$  & $5$  & $2$ & $4$  & $1$  & $0$ \\
2 & $2$ & $16$ & $14$ & $8$  & $10$ & $4$ & $7$ & $15$ & $10$ & $6$  & $12$ & $3$ & $1$ & $17$ & $9$  & $2$ & $6$  & $4$  & $0$ \\
& $3$ & $25$ & $20$ & $10$ & $14$ & $5$ & $8$ & $24$ & $17$ & $9$  & $17$ & $4$ & $1$ & $24$ & $11$ & $2$ & $12$ & $7$  & $0$ \\
& $4$ & $31$ & $24$ & $10$ & $18$ & $7$ & $8$ & $30$ & $19$ & $9$  & $23$ & $6$ & $1$ & $32$ & $13$ & $2$ & $16$ & $11$ & $0$ \\
& $5$ & $38$ & $30$ & $11$ & $21$ & $6$ & $9$ & $37$ & $25$ & $10$ & $29$ & $5$ & $1$ & $40$ & $17$ & $2$ & $18$ & $13$ & $0$ \\\hline

& $0$ & $2$  & $2$ & $2$ & $2$ & $1$ & $1$ & $1$ & $1$ & $0$ & $2$ & $2$ & $1$ & $3$ & $3$ & $2$ & $2$ & $2$ & $0$ \\
& $1$ & $10$ & $7$ & $5$ & $6$ & $3$ & $4$ & $9$ & $7$ & $3$ & $7$ & $2$ & $1$ & $11$ & $5$ & $2$ & $4$ & $3$ & $0$ \\
$3$ & $2$ & $18$ & $14$ & $8$ & $10$ & $4$ & $7$ & $17$ & $12$ & $6$ & $12$ & $3$ & $1$ & $18$ & $9$ & $2$ & $8$ & $6$ & $0$ \\
& $3$ & $26$ & $20$ & $10$ & $14$ & $5$ & $8$ & $26$ & $17$ & $9$ & $18$ & $4$ & $1$ & $25$ & $11$ & $2$ & $10$ & $7$ & $0$ \\
& $4$ & $32$ & $24$ & $10$ & $18$ & $5$ & $8$ & $30$ & $19$ & $9$ & $24$ & $4$ & $1$ & $35$ & $17$ & $2$ & $16$ & $11$ & $0$ \\
& $5$ & $41$ & $31$ & $11$ & $21$ & $6$ & $9$ & $38$ & $24$ & $10$ & $30$ & $5$ & $1$ & $42$ & $18$ & $2$ & $17$ & $13$ & $0$ \\\hline

& $0$ & $4$ & $3$ & $2$ & $1$ & $0$ & $1$ & $2$ & $1$ & $0$ & $2$ & $1$ & $1$ & $4$ & $2$ & $2$ & $1$ & $1$ & $0$\\
& 				$1$ & $12$ & $9$  & $5$  & $6$ & $3$  & $4$ & $9$  & $6$  & $3$  & $7$  & $2$ & $1$ & $13$ & $5$  & $2$ & $4$  & $3$  & $0$ \\
$6$ & 				$2$ & $19$ & $14$ & $8$  & $10$ & $4$ & $7$ & $17$ & $12$ & $6$  & $15$ & $3$ & $1$ & $19$ & $9$  & $2$ & $8$  & $4$  & $0$ \\
& 				$3$ & $26$ & $20$ & $10$ & $14$ & $5$ & $8$ & $26$ & $18$ & $9$  & $20$ & $4$ & $1$ & $27$ & $11$ & $2$ & $10$ & $7$  & $0$ \\
& 				$4$ & $32$ & $24$ & $10$ & $18$ & $5$ & $8$ & $30$ & $20$ & $9$  & $24$ & $4$ & $1$ & $35$ & $17$ & $2$ & $16$ & $11$ & $0$ \\
& 				$5$ & $41$ & $31$ & $11$ & $21$ & $6$ & $9$ & $38$ & $25$ & $10$ & $30$ & $5$ & $1$ & $42$ & $18$ & $2$ & $17$ & $13$ & $0$ \\\hline
 \end{tabular} \caption{Coefficients in the first $6t-18$ base blocks in Lemma \ref{lem-RCDF-main}}
\label{tab:coffi}
\end{table}

We take the case $h=2$, $x=0$ and $t\equiv 1\pmod{2}$ for example to show that $\F$ forms a $(2g,2,4,1)$-CDF. For convenience, for any $F\in \F$, if $x,y\in F$ and $x>y$, we call $x-y\pmod{gh}$ a {\em positive} difference from $F$, and $y-x\pmod{gh}$ a {\em negative} difference from $F$.  The collection of all positive differences (resp. negative differences) in $\Delta F$ is denoted by $\Delta^+ F$ (resp. $\Delta^- F$). Write $\Delta^+\F_r=\bigcup_{F\in\F_r} \Delta^+ F$ and $\Delta^-\F_r=\bigcup_{F\in\F_r} \Delta^- F$ for $1\leq r\leq6$. We list all positive differences from $\F_r$ with $1\leq r\leq6$ in Table \ref{tab0}, and if a positive difference is greater than $36t+1$, then we list its corresponding negative difference. We also list all differences less than $36t+1$ from the remaining 18 base blocks, written as $B_l$ with $1\leq l\leq 18$, in Table \ref{tab00}. It is readily checked that $\Delta \F=\Z_{72t+2}\setminus\{0,36t+1\}$, and hence $\F$ is a $(2g,2,4,1)$-CDF.

\begin{table}[t]\renewcommand\arraystretch{1.3}\resizebox{\linewidth}{!}{
\begin{tabular}{|c|c|c||c|c|}\hline
$\Delta^+ F_{1,i}$ & $\Delta^+{\cal F}_1$ & $\Delta^-{\cal F}_1$ & $\Delta^+ F_{2,i}$ & $\Delta^+{\cal F}_2$\\\hline
			$43t+1+i$ & $[43t+2,44t-1]\backslash\{\frac{87t+1}{2}\}$ & $[28t+3,29t]\backslash\{\frac{57t+3}{2}\}$ & $23t+1+i$ & $[23t+2,24t-1]\backslash\{\frac{47t+1}{2}\}$     \\
			$31t+1+2i$ & $[31t+3,33t-3]_{2}\backslash\{32t\}$ & & $5t+2i$ & $[5t+2,7t-4]_{2}\backslash\{6t-1\}$  \\
			$8t+2+3i$ & $[8t+5,11t-4]_{3}\backslash\{\frac{19t+1}{2}\}$ & & $8t+1+3i$ & $[8t+4,11t-5]_{3}\backslash\{\frac{19t-1}{2}\}$  \\
			$12t-i$ & $[11t+2,12t-1]\backslash\{\frac{23t+1}{2}\}$ & & $18t+1-i$ & $[17t+3,18t]\backslash\{\frac{35t+3}{2}\}$  \\
			$35t-1-2i$ & $[33t+3,35t-3]_{2}\backslash\{34t\}$ & & $15t-2i$ & $[13t+4,15t-2]_{2}\backslash\{14t+1\}$  \\
			$23t-1-i$ & $[22t+1,23t-2]\backslash\{\frac{45t-1}{2}\}$ & & $3t+1+i$ & $[3t+2,4t-1]\backslash\{\frac{7t+1}{2}\}$  \\\hline
			$\Delta F^+_{3,i}$ & $\Delta{\cal F}^+_3$ & $\Delta{\cal F}^-_3$ & $\Delta F^+_{4,i}$ & $\Delta{\cal F}^+_4$     \\\hline
			$41t+i$ & $[41t+1,42t-2]\backslash\{\frac{83t-1}{2}\}$ & $[30t+4,31t+1]\backslash\{\frac{61t+5}{2}\}$ & $35t+1+i$ & $[35t+2,36t-1]\backslash\{\frac{71t+1}{2}\}$     \\
			$25t+1+2i$ & $[25t+3,27t-3]_{2}\backslash\{26t\}$ & & $5t+1+2i$ & $[5t+3,7t-3]_{2}\backslash\{6t\}$  \\
			$8t+3i$ & $[8t+3,11t-6]_{3}\backslash\{\frac{19t-3}{2}\}$ & & $1+3i$ & $[4,3t-5]_{3}\backslash\{\frac{3t-1}{2}\}$  \\
			$16t-1-i$ & $[15t+1,16t-2]\backslash\{\frac{31t-1}{2}\}$ & & $30t-i$ & $[29t+2,30t-1]\backslash\{\frac{59t+1}{2}\}$  \\
			$33t-2i$ & $[31t+4,33t-2]_{2}\backslash\{32t+1\}$ & & $35t-2i$ & $[33t+4,35t-2]_{2}\backslash\{34t+1\}$  \\
			$17t+1-i$ & $[16t+3,17t]\backslash\{\frac{33t+3}{2}\}$ & & $5t-i$ & $[4t+2,5t-1]\backslash\{\frac{9t+1}{2}\}$  \\\hline
			$\Delta F^+_{5,i}$ & $\Delta{\cal F}^+_5$ & $\Delta{\cal F}^-_5$ & $\Delta F^+_{6,i}$ & $\Delta{\cal F}^+_6$     \\\hline
			$47t+2+i$ & $[47t+3,48t]\backslash\{\frac{95t+3}{2}\}$ & $[24t+2,25t-1]\backslash\{\frac{49t+1}{2}\}$ & $21t+1+i$ & $[21t+2,22t-1]\backslash\{\frac{43t+1}{2}\}$     \\
			$19t+2+2i$ & $[19t+4,21t-2]_{2}\backslash\{20t+1\}$ & & $13t+1+2i$ & $[13t+3,15t-3]_{2}\backslash\{14t\}$  \\
			$2+3i$ & $[5,3t-4]_{3}\backslash\{\frac{3t+1}{2}\}$ & & $3i$ & $[3,3t-6]_{3}\backslash\{\frac{3t-3}{2}\}$  \\
			$28t-i$ & $[27t+2,28t-1]\backslash\{\frac{55t+1}{2}\}$ & & $8t-i$ & $[7t+2,8t-1]\backslash\{\frac{15t+1}{2}\}$  \\
			$47t-2i$ & $[45t+4,47t-2]_{2}\backslash\{46t+1\}$ & $[25t+4,27t-2]_2\backslash\{26t+1\}$ & $21t+1-2i$ & $[19t+5,21t-1]_{2}\backslash\{20t+2\}$  \\
			$19t-i$ & $[18t+2,19t-1]\backslash\{\frac{37t+1}{2}\}$ & & $13t+1-i$ & $[12t+3,13t]\backslash\{\frac{25t+3}{2}\}$  \\\hline
\end{tabular}
}
\caption{Differences from the first $6t-18$ base blocks in Lemma \ref{lem-RCDF-main} when $h=2$, $x=0$ and $t$ is odd} \label{tab0}
\end{table}

\begin{table}[b]\renewcommand\arraystretch{1.15}\resizebox{\linewidth}{!}{
\begin{tabular}{|c|c|c|}\hline
 	$j$ & differences less than $36t+1$ from $B_{2j-1}$ & differences less than $36t+1$ from $B_{2j}$ \\\hline
 	$1$ &
 	$\{1,11t-3,11t-2,30t+2,30t+3,31t+2\}$ &
 	$\{2,3t-3,3t-1,8t+2,11t-1,11t+1\}$
 	 \\[0.25ex]\hline
 	$2$ & $\{\frac{3t-3}{2},23t-1,28t+2,\frac{43t+1}{2},\frac{59t+1}{2},21t+1\}$ &
 	$\{3t,33t+1,12t,30t+1,15t,27t+1\}$\\[0.25ex]\hline
 	$3$ &
 	$\{3t+1,7t+1,19t+3,4t,16t+2,12t+2\}$ &
 	$\{\frac{9t+1}{2},8t+1,\frac{3t-1}{2},\frac{7t+1}{2},6t,\frac{19t+1}{2}\}$\\[0.25ex]\hline
 	$4$ &
 	$\{5t,21t,28t,16t,7t,23t\}$ &
 	$\{5t+1,\frac{55t+1}{2},\frac{15t+1}{2},\frac{45t-1}{2},\frac{25t+3}{2},35t+1\}$\\[0.25ex]\hline
 	$5$ &
 	$\{6t-1,19t+1,35t,13t+2,16t-1,29t+1\}$ &
 	$\{7t-2,35t-1,20t+2,28t+1,17t+1,27t\}$\\[0.25ex]\hline
 	$6$ &
 	$\{8t,33t,14t,22t,25t,25t+2\}$ &
 	$\{\frac{19t-1}{2},\frac{61t+5}{2},19t+2,\frac{57t+3}{2},32t,\frac{23t+1}{2}\}$\\[0.25ex]\hline
 	$7$ &
 	$\{11t,34t+1,14t+1,23t+1,24t,25t+1\}$ &
 	$\{12t+1,15t-1,33t+2,3t-2,27t-1,24t+1\}$\\[0.25ex]\hline
 	$8$ &
 	$\{\frac{31t-1}{2},\frac{71t+1}{2},\frac{37t+1}{2},20t+1,34t,18t+1\}$&
 	$\{16t+1,\frac{35t+3}{2},30t,\frac{3t+1}{2},\frac{49t+1}{2},26t+1\}$\\[0.25ex]\hline
 	$9$ &
 	$\{\frac{33t+3}{2},26t,7t-1,\frac{47t+1}{2},\frac{19t-3}{2},33t-1\}$&
 	$\{19t,32t+1,36t,13t+1,4t+1,17t+2\}$\\[0.25ex]\hline
\end{tabular}
}
\caption{Differences less than $36t+1$ from the remaining $18$ base blocks in Lemma \ref{lem-RCDF-main} when $h=2$, $x=0$ and $t$ is odd} \label{tab00}
\end{table}

One can examine all the other cases similarly. To facilitate the reader to check the correctness of our results, we provide a computer code written by GAP \cite{GAP4}. The reader can get a copy of the computer code from \cite{Checking}.

We remark that the structure of the first $6t-18$ base blocks in this proof is exactly the same as the structure of the first $6t-18$ base blocks in the proof for the existence of $(gh,h,4,1)$-CDFs with $h\in\{1,4\}$ in \cite[Lemmas 4, 5 and Theorem 5]{Zhang1}. We just need to choose appropriate parameters $a_j,b_j,c_j,d_j,e_j,f_j$ by hand and then find the remaining $18+x$ base blocks by computer search such that $\F$ forms our required difference families. It often took us around a half day to choose an appropriate set of $a_j,b_j,c_j,d_j,e_j,f_j$ for each given $h$ and $x$, and then the time used to search for the remaining $18+x$ base blocks by an ordinary personal computer was ranged from one day to three days.
\end{proof}

The following lemma will be used to show the existence of $(gh,h,4,1)$-CDFs with $h\in\{2,3\}$ and small $g$'s later. We outline its proof for completeness.

\begin{Lemma}\label{lem:classic}
\begin{enumerate}
\item[$(1)$] {\rm \cite[Lemma 6.10]{Brouwer}} There exists a $(2p,2,4,1)$-CDF for any prime $p\equiv1\pmod{6}$.
\item[$(2)$] {\rm \cite{Moore}} There exists a $(3p,3,4,1)$-CDF for any prime $p\equiv1\pmod{4}$.
\end{enumerate}
\end{Lemma}

\begin{proof}
Let $\omega$ be a primitive root of unity in the finite field $\mathbb{F}_p$. Let $\varepsilon=\omega^{(p-1)/6}$ if $p\equiv1\pmod{6}$, and $\varepsilon=\omega^{(p-1)/4}$ if $p\equiv1\pmod{4}$. Then
$$\{(0,0),(1,\omega^i),(1,\varepsilon^2\cdot \omega^i),(1,\varepsilon^4\cdot \omega^i)\},\ \ \ \ 0\leq i<(p-1)/6,$$
forms a $(2p,2,4,1)$-CDF over $\Z_2\times \Z_p\cong \Z_{2p}$ for any prime $p\equiv1\pmod{6}$, and
$$\{(0,\omega^i),(0,\varepsilon^2\cdot \omega^i),(1,\varepsilon\cdot \omega^i),(1,\varepsilon^3\cdot \omega^i)\},\ \ \ \ 0\leq i<(p-1)/4,$$
forms a $(3p,3,4,1)$-CDF over $\Z_3\times \Z_p\cong \Z_{3p}$ for any prime $p\equiv1\pmod{4}$.
\end{proof}

Applying a standard recursive construction (see Construction \ref{recursion} below) together with the use of the above lemma, Yin \cite{Yin1} established the following results implicitly.

\begin{Lemma}\label{lem:h=2}
\begin{enumerate}
\item[$(1)$] {\rm \cite[Theorem 4.9]{Yin1}} There exists a $(2g,2,4,1)$-CDF for any positive integer $g$ such that each prime factor of $g$ is $1$ modulo $6$.
\item[$(2)$] {\rm \cite[Theorem 4.10]{Yin1}} There exists a $(3g,3,4,1)$-CDF for any positive integer $g$ such that each prime factor of $g$ is $1$ modulo $4$.
\end{enumerate}
\end{Lemma}

\begin{Lemma}\label{lem:h=2-small}
There exists a $(2g,2,4,1)$-CDF for any integer $g\equiv 1\pmod{6}$ and $7\leq g\leq 103$.
\end{Lemma}

\begin{proof}
For $g\in\{25,55,85\}$, we list all the base blocks below.
\begin{center}
\begin{tabular}{llllll}
$g=25$: & $\{0,1,3,12\}$,&$\{0,4,18,26\}$,&$\{0,6,23,43\}$,&$\{0,10,15,31\}$. \\

$g=55$: & $\{0,1,4,50\}$,&$\{0,2,7,36\}$,&$\{0,6,15,100\}$,&$\{0,8,31,99\}$,&$\{0,13,37,90\}$,\\
&$\{0,14,54,72\}$,&$\{0,17,47,82\}$,&$\{0,26,48,69\}$,&$\{0,27,59,71\}$. \\

$g=85$: & $\{0,1,4,62\}$,&$\{0,2,7,22\}$,&$\{0,6,14,142\}$,&$\{0,9,21,72\}$,&$\{0,10,23,55\}$,\\
&$\{0,11,29,93\}$,&$\{0,16,52,91\}$,&$\{0,17,73,116\}$,&$\{0,19,84,122\}$,&$\{0,24,59,92\}$,\\
&$\{0,27,74,144\}$,&$\{0,31,80,120\}$,&
$\{0,44,69,110\}$,&$\{0,46,76,133\}$.
\end{tabular}
\end{center}
For all the other cases of $g$, apply Lemma \ref{lem:h=2}.
\end{proof}

\begin{Lemma}\label{lem:h=3-small}
There exists a $(3g,3,4,1)$-CDF for any integer $g\equiv 1\pmod{4}$, $5\leq g\leq 69$ and $g\neq 9$.
\end{Lemma}

\begin{proof}
For $g\in\{21,33,45,49,57,69\}$, we list all the base blocks below.
\begin{center}\tabcolsep 0.07in
\begin{tabular}{llllll}
$g=21$: & $\{0,2,18,27\}$,&$\{0,3,10,15\}$,&$\{0,8,14,37\}$,&$\{0,11,30,31\}$,&$\{0,24,28,41\}$. \\

$g=33$: & $\{0,1,3,40\}$,&$\{0,4,12,86\}$,&$\{0,5,16,36\}$,&$\{0,6,44,70\}$,&$\{0,10,28,85\}$,\\
&$\{0,15,56,65\}$&$\{0,19,46,67\}$,&$\{0,22,45,52\}$. \\

$g=45$: & $\{0,1,5,38\}$,&$\{0,2,8,60\}$,&$\{0,3,14,21\}$,&$\{0,9,29,120\}$,&$\{0,10,59,84\}$,\\
&$\{0,12,39,107\}$,&$\{0,16,46,87\}$,&$\{0,17,73,109\}$,&$\{0,23,65,78\}$,&$\{0,32,63,85\}$,\\
&$\{0,35,54,101\}$.\\

$g=49$: & $\{0,1,3,43\}$,&$\{0,19,73,101\}$,&$\{0,6,14,38\}$,&$\{0,7,17,68\}$,&$\{0,35,69,102\}$,\\
&$\{0,4,9,22\}$,&$\{0,20,92,108\}$,&$\{0,21,57,87\}$,&$\{0,29,77,100\}$,&$\{0,31,58,122\}$,\\
&$\{0,11,26,63\}$,&$\{0,50,62,103\}$.\\

$g=57$: & $\{0,1,3,55\}$,&$\{0,44,81,108\}$,&$\{0,5,12,105\}$,&$\{0,8,17,153\}$,&$\{0,11,24,80\}$,\\
&$\{0,14,29,99\}$,&$\{0,16,50,83\}$,&$\{0,21,79,118\}$,&$\{0,25,48,68\}$,&$\{0,40,62,135\}$,\\
&$\{0,4,10,42\}$,&$\{0,46,87,106\}$,&$\{0,47,77,122\}$,&$\{0,51,82,110\}$. \\

$g=69$: & $\{0,1,5,124\}$,&$\{0,15,34,108\}$,&$\{0,35,110,156\}$,&$\{0,55,98,137\}$,&$\{0,11,24,193\}$,\\
&$\{0,2,8,66\}$,&$\{0,18,59,85\}$,&$\{0,45,105,135\}$,&$\{0,20,77,171\}$,&$\{0,28,89,139\}$,\\
&$\{0,3,12,170\}$,&$\{0,44,71,175\}$,&$\{0,22,100,142\}$,&$\{0,48,79,160\}$,&$\{0,52,73,153\}$,\\
&$\{0,7,17,191\}$,&$\{0,63,92,154\}$.
\end{tabular}
\end{center}
For all the other cases of $g$, apply Lemma \ref{lem:h=2}.
\end{proof}

The proof of Lemma 2.6 in \cite{KChen} implies that if there exists a skew starter in $\Z_g$, then there exists a $(6g,6,4,1)$-CDF. Since a skew starter in $\Z_g$ exists for any $g$ such that $\gcd(g,150)=1$ or $25$ by \cite[Theorem 1.2]{KChen}, and a skew starter in $\Z_{35}$ exists by \cite[Lemma 2.4]{KChen}, we have the following result, (1) of which was first stated explicitly in \cite[Lemma 2.6]{gy}.

\begin{Lemma}\label{lem:h=6} {\rm \cite{KChen}}
\begin{enumerate}
\item[$(1)$] There exists a $(6g,6,4,1)$-CDF for any positive integer $g$ such that $\gcd(g,150)=1$ or $25$.
\item[$(2)$] There exists a $(210,6,4,1)$-CDF.
\end{enumerate}
\end{Lemma}

We remark that Buratti \cite[Theorem 4.1]{MBuratti2} provided a nice explicit construction for the existence of a $(6p,6,4,1)$-CDF for any odd prime $p\neq 5$ according to whether $p$ is or is not a Fermat prime.

\begin{Lemma}\label{lem:h=6-small}
There exists a $(6g,6,4,1)$-CDF for any odd integer $7\leq g\leq 35$.
\end{Lemma}

\begin{proof}
For $g\in\{9,15,27\}$, a $(6g,6,4,1)$-CDF exists by \cite[Lemma 4.2]{gy}. For $g\in\{21,33\}$, we list all the base blocks below.
\begin{center}
\begin{tabular}{llllll}
$g=21$: & $\{0,19,62,79\}$,&$\{0,2,8,78\}$,&$\{0,3,12,35\}$,&$\{0,13,31,65\}$,&$\{0,15,22,115\}$,\\
&$\{0,16,55,106\}$,&$\{0,1,5,46\}$,&$\{0,24,38,96\}$,&$\{0,25,53,82\}$,&$\{0,27,37,86\}$.\\

$g=33$: &$\{0,49,88,125\}$,&$\{0,2,9,74\}$,&$\{0,25,83,106\}$,&$\{0,17,35,69\}$,&$\{0,38,94,138\}$,\\
&$\{0,14,29,170\}$,&$\{0,4,12,82\}$,&$\{0,21,45,130\}$,&$\{0,3,13,179\}$,&$\{0,30,105,148\}$,\\
&$\{0,31,86,127\}$,&$\{0,1,6,54\}$,&$\{0,40,91,137\}$,&$\{0,11,27,90\}$,&$\{0,47,114,134\}$,\\
&$\{0,59,95,121\}$.
\end{tabular}
\end{center}
For all the other cases of $g$, apply Lemma \ref{lem:h=6}.
\end{proof}

Now we are in a position to give a proof of Theorem \autoref{thm:main}.

\begin{proof}[{\bf Proof of Theorem \autoref{thm:main}}]
By Lemma \ref{lem:nece-DF}, a $(gh,h,4,1)$-CDF exists only if $gh\equiv h\pmod{12}$ and $g\geq 4$. It is known that there is no $(gh,h,4,1)$-CDF for $(g,h)\in\{(9,3),(5,6)\}$ (cf. \cite{MBuratti2,Survey}). Combining the results of Lemmas \ref{lem-RCDF-main}, \ref{lem:h=2-small}, \ref{lem:h=3-small} and \ref{lem:h=6-small}, one can complete the proof.
\end{proof}

\subsection{General cases}\label{sec-recursions}

Difference matrices have proved to be a very useful tool for constructing cyclic relative difference families. Let $(G,+)$ be a finite group of order $v$ and $k\geq 2$ be an integer. A $(v,k,1)$-\emph{difference matrix} (DM) over $G$, or briefly $(G,k,1)$-DM, is a $k\times v$ matrix $D=(d_{ij})$ with entries from $G$ such that for any two distinct rows $x$ and $y$, the multiset of differences $\{d_{xj}-d_{yj}:1\leq j\leq v\}$ contains each element of $G$ exactly once. A $(\mathbb{Z}_v,k,1)$-DM is often called \emph{cyclic} and denoted by a $(v,k,1)$-CDM.

The following construction is a standard recursive construction for cyclic relative difference families.

\begin{Construction}\rm\cite{Jungnickel}\label{recursion}
If there exist a $(gn,g,k,1)$-CDF and an $(h,k,1)$-CDM, then there exists a $(gnh,gh,k,1)$-CDF. Furthermore, if there exists a $(gh,g,k,1)$-CDF, then there exists a $(ghn,g,k,1)$-CDF.
\end{Construction}

Evans \cite{Evans02} constructed a $(v,4,1)$-CDM for any odd integer $v\geq 5$ and $v$ is not divisible by $9$. Zhang, Feng and Wang \cite{zfw} solved the case $v\equiv 9\pmod{18}$.

\begin{Lemma}\label{cdm}{\rm \cite{Evans02,zfw}}
A $(v,4,1)$-CDM exists if and only if $v\geq 5$ is odd and $v\neq 9$.
\end{Lemma}

By Lemma \ref{lem:nece-DF}, a $(4h,h,4,1)$-CDF exists only if $h\equiv 0\pmod{4}$. Since a  $(4h,h,4,1)$-CDF yields a strictly cyclic group divisible design of type $h^4$ with block size four, which can produce an $(h,4,1)$-CDM (cf. \cite[Theorem VIII.3.6]{bjl}), it follows from Lemma \ref{cdm} that we have the following lemma.

\begin{Lemma}\label{nonexist-DF}
There does not exist a $(4h,h,4,1)$-CDF for any positive integer $h$.
\end{Lemma}

\begin{Theorem}\label{thm:RDF-general}
There exists a $(gh,h,4,1)$-CDF if and only if $gh\equiv h\pmod{12}$ and $g\geq 5$ except for $(g,h)\in\{(9,3),(5,6)\}$ and except possibly for
\begin{enumerate}
\item[$(1)$] $h\equiv 1,5,7,11\pmod{12}$ and $g=25$,
\item[$(2)$] $h\in\{9,27\}$ and $g\equiv1\pmod{4}$,
\item[$(3)$] $h\equiv3,9\pmod{12}$ and $g=9$,
\item[$(4)$] $h\equiv 8,16\pmod{24}$ and $g\equiv1\pmod{3}$,
\item[$(5)$] $h\equiv 4,20\pmod{24}$ and $g=7$,
\item[$(6)$] $h\equiv 6\pmod{12}$ and $g=5$,
\item[$(7)$] $h=18$ and $g\equiv 1\pmod{2}$,
\item[$(8)$] $h=54$ and $g\equiv 3,5\pmod{6}$,
\item[$(9)$] $h\equiv 0\pmod{24}$ and $g\geq 5$,
\item[$(10)$] $h\equiv 12\pmod{24}$, $g\equiv 0,2\pmod{3}$ and $g\geq 5$,
\item[$(11)$] $h\in\{12,36\}$, $g\equiv 1\pmod{3}$ and $g\geq 10$.
\end{enumerate}
\end{Theorem}

\begin{proof}
By Lemmas \ref{lem:nece-DF} and \ref{nonexist-DF}, a $(gh,h,4,1)$-CDF exists only if $g\geq 5$ and $h(g-1)\equiv 0\pmod{12}$, that is, (i) $h\equiv 1,5,7,11\pmod{12}$ and $g\equiv1\pmod{12}$, (ii) $h\equiv 2,10\pmod{12}$ and $g\equiv1\pmod{6}$, (iii) $h\equiv 3,9\pmod{12}$ and $g\equiv1\pmod{4}$, (iv) $h\equiv 4,8\pmod{12}$ and $g\equiv1\pmod{3}$, (v) $h\equiv 6\pmod{12}$ and $g\equiv1\pmod{2}$, (vi) $h\equiv 0\pmod{12}$ and $g\geq 5$.

For $h\equiv 1,5,7,11\pmod{12}$, $g\equiv1\pmod{12}$ and $g\neq 25$, apply Construction \ref{recursion} with a $(g,1,4,1)$-CDF (from Theorem \ref{thm:cyclic design k=4}) and an $(h,4,1)$-CDM (from Lemma \ref{cdm}) to obtain a $(gh,h,4,1)$-CDF.

For $h\equiv 2,10\pmod{12}$ and $g\equiv1\pmod{6}$, apply Construction \ref{recursion} with a $(2g,2,4,1)$-CDF (from Theorem \ref{thm:main}) and an $(h/2,4,1)$-CDM.

For $h\equiv 3,9\pmod{12}$, $g\equiv1\pmod{4}$, $h\not\in\{9,27\}$ and $g\neq 9$, apply Construction \ref{recursion} with a $(3g,3,4,1)$-CDF (from Theorem \ref{thm:main}) and an $(h/3,4,1)$-CDM.

For $h\equiv 4,20\pmod{24}$, $g\equiv1\pmod{3}$ and $g\geq 10$, apply Construction \ref{recursion} with a $(4g,4,4,1)$-CDF (from Theorem \ref{thm:cyclic design k=4}) and an $(h/4,4,1)$-CDM.

For $h\equiv 6\pmod{12}$, $g\equiv1\pmod{2}$, $h\not\in\{18,54\}$ and $g\geq 7$, apply Construction \ref{recursion} with a $(6g,6,4,1)$-CDF (from Theorem \ref{thm:main}) and an $(h/6,4,1)$-CDM. For $h=54$ and $g\equiv1\pmod{6}$, apply Construction \ref{recursion} with a $(2g,2,4,1)$-CDF (from Theorem \ref{thm:main}) and a $(27,4,1)$-CDM.

For $h\equiv 12\pmod{24}$, $g\equiv1\pmod{3}$, $h\geq 60$ and $g\geq 10$, apply Construction \ref{recursion} with a $(4g,4,4,1)$-CDF (from Theorem \ref{thm:cyclic design k=4}) and an $(h/4,4,1)$-CDM. For $h\equiv 12\pmod{24}$, $g=7$, $h\geq 60$ and $h\neq 108$, start from an $(84,12,4,1)$-CDF whose base blocks are listed below:
\begin{center}
\begin{tabular}{llllll}
$\{0,1,3,9\}$,&
$\{0,4,15,37\}$,&
$\{0,5,24,55\}$,&
$\{0,10,30,46\}$,&
$\{0,12,25,57\}$,&
$\{0,17,40,58\}$,
\end{tabular}
\end{center}
and then apply Construction \ref{recursion} with an $(h/12,4,1)$-CDM. By \cite[Lemma 3.10]{CFM}, there exists a $(36\times 7,36,4,1)$-CDF. By \cite[Theorem 3.9]{CFM}, there exists a $(108\times 7,108,4,1)$-CDF.
\end{proof}

\section{Applications}\label{sec-3}

\subsection{Cyclic group divisible designs}\label{gdd}

A {\em group divisible design} (GDD) is a triple $(V,\G,\B)$, where $V$ is a set of {\em points}, $\G$ is a partition of $V$ into {\em groups}, and $\B$ is a family of subsets of $V$ (called {\em blocks}) such that every block intersects any given group in at most one point, and every two points belonging to distinct groups are contained in exactly one block. A $k$-GDD of {\em type} $h^g$ is a GDD in which the size of each block is $k$ and there are $g$ groups of size $h$. For more details on group divisible designs, the reader is referred to \cite{bjl1,ge}.

An {\em automorphism} of a GDD ($V,{\cal G},{\cal B}$) is a permutation on $V$ leaving ${\cal G}$ and ${\cal B}$ invariant, respectively. A $k$-GDD of type $h^g$ is said to be {\em cyclic} if it admits an automorphism consisting of a cycle of length $gh$. Without loss of generality, we regard $\Z_{gh}$ as an automorphism group of a cyclic $k$-GDD of type $h^g$. If the stabilizer of any block of a cyclic GDD of type $h^g$ in $\Z_{gh}$ is trivial, then the GDD is called {\em strictly cyclic}. The blocks of a strictly cyclic $k$-GDD of type $h^g$ can be partitioned into orbits under ${\mathbb Z}_{gh}$. Choose any fixed block from each orbit and then call them base blocks. These base blocks form a $(gh,h,k,1)$-CDF. Therefore, a strictly cyclic $k$-GDD of type $h^g$ yields a $(gh,h,k,1)$-CDF, and vice versa. Combining Theorem \ref{thm:RDF-general}, we have the following theorem.

\begin{Theorem}\label{thm:CGDD-general}
There exists a strictly cyclic $4$-GDD of type $h^g$ if and only if $gh\equiv h\pmod{12}$ and $g\geq 5$ except for $(g,h)\in\{(9,3),(5,6)\}$ and except possibly for
\begin{enumerate}
\item[$(1)$] $h\equiv 1,5,7,11\pmod{12}$ and $g=25$,
\item[$(2)$] $h\in\{9,27\}$ and $g\equiv1\pmod{4}$,
\item[$(3)$] $h\equiv3,9\pmod{12}$ and $g=9$,
\item[$(4)$] $h\equiv 8,16\pmod{24}$ and $g\equiv1\pmod{3}$,
\item[$(5)$] $h\equiv 4,20\pmod{24}$ and $g=7$,
\item[$(6)$] $h\equiv 6\pmod{12}$ and $g=5$,
\item[$(7)$] $h=18$ and $g\equiv 1\pmod{2}$,
\item[$(8)$] $h=54$ and $g\equiv 3,5\pmod{6}$,
\item[$(9)$] $h\equiv 0\pmod{24}$ and $g\geq 5$,
\item[$(10)$] $h\equiv 12\pmod{24}$, $g\equiv 0,2\pmod{3}$ and $g\geq 5$,
\item[$(11)$] $h\in\{12,36\}$, $g\equiv 1\pmod{3}$ and $g\geq 10$.
\end{enumerate}
\end{Theorem}

\subsection{1-rotational Steiner systems}\label{rotational}

A $k$-GDD of type $1^v$ is often called a {\em Steiner system} and denoted by an S$(2,k,v)$. An S$(2,k,v)$ is said to be {\em $1$-rotational} if it admits an automorphism consisting of one fixed point and a cycle of length $v-1$. A 1-rotational S$(2,k,v)$ defined on $X$ may be assumed to have $X=\Z_{v-1}\cup\{\infty\}$ with the 1-rotational automorphism given by $i\rightarrow i+1\pmod{v-1}$ for $i\in \Z_{v-1}$ and $\infty \rightarrow \infty$.

Phelps and Rosa \cite{Phelps} showed that a 1-rotational S$(2,3,v)$ exists if and only if $v\equiv 3,9\pmod{24}$. It was conjectured in \cite[Section 3.2]{Survey} that a 1-rotational S$(2,4,v)$ exists if and only if $v\equiv4\pmod{12}$ and $v\neq 28$. Applying Theorem \ref{thm:main}, we can confirm this conjecture.

\begin{Theorem}\label{thm:1-rotational}
There exists a 1-rotational S$(2,4,v)$ if and only if $v\equiv4\pmod{12}$ and $v\neq 28$.
\end{Theorem}

\begin{proof}
The necessity follows from \cite[Section 3.2]{Survey}. To show the sufficiency, take a $(v-1,3,4,1)$-CDF, $\F$, with $v\equiv4\pmod{12}$ and $v\neq 28$ from Theorem \ref{thm:main}. Then it is readily checked that $\{\{f_1+t,f_2+t,f_3+t,f_4+t\}:F=\{f_1,f_2,f_3,f_4\}\in \F,t\in \Z_{v-1}\}\cup\{\{t,\frac{v-1}{3}+t,\frac{2(v-1)}{3}+t,\infty\}:0\leq t<\frac{v-1}{3}\}$ forms a 1-rotational S$(2,4,v)$.
\end{proof}

For more information on 1-rotational Steiner systems, the reader is referred to \cite{ab,Jimbo1}.

\subsection{Optimal cyclic constant-weight codes}\label{sec-5-2}

A $q$-ary code $\mathcal{C}$ of length $n$ is a set of vectors in $\{0,1,\ldots,q-1\}^n$. An element in $\mathcal{C}$ is called a {\em codeword}. For two codewords $u=(u_i)_{0\leq i\leq n-1}$ and $v=(v_i)_{0\leq i\leq n-1}$, the {\em Hamming distance} between $u$ and $v$ is $d_H(u,v)=|\{0\leq i\leq n-1:u_i-v_i\neq 0\}|$. The {\em Hamming weight} of $u$ is the
Hamming distance between $u$ and the zero vector 0. A code is said to be of {\em constant-weight} $w$ if the Hamming weight of each codeword is a constant $w$. A code $\mathcal{C}$ is said to have {\em minimum distance} $d$ if $d_H(u,v)\geq d$ for all distinct $u,v\in \mathcal{C}$. A $q$-ary code of length $n$, minimum distance $d$ and constant-weight $w$ is denoted by an $(n,d,w)_q$ code.

An $(n,d,w)_q$ code $\mathcal{C}$ is said to be {\em cyclic} if $(u_0,u_1,\dots,u_{n-1}) \in \mathcal{C}$ implies $(u_1,u_2,\ldots,u_{n-1},u_0)\in \mathcal{C}$. We use the notation $CA_q(n,d,w)$ to denote the largest possible number of codewords among all cyclic $(n,d,w)_q$ codes with given parameters $n,d,w$ and $q$. A cyclic $(n,d,w)_q$ code is called {\em optimal} if it has $CA_q(n,d,w)$ codewords.

Lan, Chang and Wang \cite{Lan2} gave an equivalent definition of a cyclic $(n,d,w)_q$ code from a set-theoretic perspective.
Let $\mathcal{C}$ be a cyclic $(n,d,w)_q$ code. For each $u=(u_i)_{0\leq i\leq n-1}\in\mathcal{C}$, construct a $w$-subset $B_u$ of $\Z_n\times \{1,2,\ldots,q-1\}$ such that $(i,u_i)\in B_u$ if and only if $u_i$ is nonzero. Let $\mathcal{C}'=\{B_{u}:u\in \mathcal{C}\}$. For any codeword $B=\{(i_1,a_1),(i_2,a_2),\ldots,(i_w,a_w)\}\in \mathcal{C^{'}}$, define the {\em projection} $p(B)$ of $B$ to be the set $p(B)=\{i_1,i_2,\ldots,i_w\}$. For any $B_u,B_v \in \mathcal{C^{'}}$, define the distance $d_H(B_u,B_v)$ to be $d_H(u,v)$. It is readily checked that $d_H(B_u,B_v)=2w-|p(B_u)\cap p(B_v)|-|B_u \cap B_v|$. Then a cyclic $(n,d,w)_q$ code $\mathcal{C}$ is a set of $w$-subsets of $\Z_n\times \{1,2,\ldots,q-1\}$ satisfying that (1) for any $B\in {\cal C}$, the cardinality of $p(B)$ is $w$; (2) $d_H(A,B)\ge d$ for any distinct codewords $A,B\in {\cal C}$; (3) if $B\in {\cal C}$, then $B+1\in {\cal C}$, where $B+1=\{(x+1\mod{n},a):(x,a)\in B\}$.

\begin{Lemma}\rm\label{lem-5-1}\cite[Lemma 3.4]{Lan1}
For any positive integer $n$,
$$
CA_3(n,6,4)\leq C(n):=\left\{
\begin{array}{ll}
n\lfloor\frac{n-1}{6}\rfloor,  & \mbox{if}\ n\ \mbox{is odd or } n\equiv 2\ ({\rm mod}\ 6);\\

n\lfloor\frac{n-1}{6}\rfloor+\frac{n}{2}, & \mbox{if}\ n \equiv 0,4\ ({\rm mod}\ 6).
\end{array}\right.
$$
\end{Lemma}

By \cite[Construction 5.3]{Lan1}, if there exists a $(gh,h,4,1)$-CDF with $h\in\{1,2,3\}$, then there exists an optimal cyclic $(gh,6,4)_3$ code with $2gh\lfloor(gh-1)/12\rfloor=gh(gh-h)/6$ codewords. It follows from Theorems \ref{thm:cyclic design k=4} and \ref{thm:main} that an optimal cyclic $(n,6,4)_3$ code with $C(n)$ codewords exists for any $n\equiv 1,2,3\pmod{12}$ and $n\not\in\{25,27\}$. For $n\in\{25,27\}$, an optimal cyclic $(n,6,4)_3$ code with $C(n)$ codewords exists by \cite[Lemma 3.5]{Lan1}. By \cite[Theorems 5.7 and 5.8]{Lan1}, there exists an optimal cyclic $(n,6,4)_3$ code with $C(n)$ codewords for any $n\equiv 6\pmod{12}$ or $n\equiv 0\pmod{24}$. Therefore, we have the following result.

\begin{Theorem}\label{thm:cyclic code}
There exists an optimal cyclic $(n,6,4)_3$ code with $C(n)$ codewords for any positive integer $n\equiv 1,2,3,6\pmod{12}$ or $n\equiv 0\pmod{24}$.
\end{Theorem}

\subsection{Optimal optical orthogonal codes}\label{sec-5-3}

The study of optical orthogonal codes was first motivated by an application in a fiber optic code-division multiple access channel which requires binary sequences with good correlation properties (cf. \cite{csw}).

Let $v$ and $k$ be positive integers. A $(v, k, 1)$ {\em optical orthogonal code} (briefly $(v, k, 1)$-OOC), ${\cal C}$, is a family of $(0, 1)$-sequences (called {\em codewords}) of {\em length} $v$ and {\em weight} $k$ satisfying the following two properties (all subscripts are reduced modulo $v$):
\begin{enumerate}
\item[(1)]
${\sum}_{0 \le t \le v-1}x_tx_{t+i} \le 1$ for any ${\bf x} = (x_0, x_1, \ldots, x_{v-1}) \in {\cal C}$ and any integer $i\not\equiv 0 \pmod{v}$ ({\em the auto-correlation property});
\item[(2)]
${\sum}_{0 \le t \le v-1}x_ty_{t+i} \le 1$ for any ${\bf x} = (x_0, x_1,$ $ \ldots,$ $ x_{v-1}) \in {\cal C}$, ${\bf y} = (y_0, y_1, \ldots, y_{v-1}) \in {\cal C}$ with ${\bf x} \neq {\bf y}$, and any integer $i$ ({\em the cross-correlation property}).
\end{enumerate}

A convenient way of viewing OOCs is from a set-theoretic perspective. A $(v,k,1)$-OOC can be defined as a set of $k$-subsets (corresponding to codewords) of $\mathbb{Z}_v$ whose list of differences does not contain repeated elements (cf. \cite[Theorem 2.1]{Yin1}). A $(v,k,1)$-OOC is said to be {\em J-optimal} if the size of the set of missing differences (in $\Z_v$) of this OOC is less than or equal to $k(k-1)$, and hence a J-optimal $(v,k,1)$-OOC contains $\lfloor \frac{v-1}{k\left(k-1\right)}\rfloor$ codewords. $\lfloor \frac{v-1}{k\left(k-1\right)}\rfloor$ is an upper bound of the number of codewords of a $(v,k,1)$-OOC, and is often referred to as {\em the Johnson bound}.

\begin{Theorem}\label{thm:OOC}
There exists a J-optimal $(v,4,1)$-OOC for any positive integer $v\equiv 1,2,3,4,6\pmod{12}$ or $v\equiv 0\pmod{24}$ with the only definite exception of $v=25$.
\end{Theorem}

\begin{proof}
Since a $(gh,h,4,1)$-CDF with $h\in\{1,2,3,4,6\}$ contains $\frac{gh-h}{12}=\lfloor\frac{gh-1}{12}\rfloor$ base blocks, every CDF from Theorems \ref{thm:cyclic design k=4} and \ref{thm:main} implies the existence of a J-optimal $(gh,4,1)$-OOC. Therefore, there exists a J-optimal $(v,4,1)$-OOC for any $v\equiv 1,2,3,4,6\pmod{12}$ and $v\not\in\{16,18,25,27,28,30\}$. By \cite[Theorem 1.6]{ab04}, there exists a J-optimal $(v,4,1)$-OOC for $v\in\{16,18,27,28,30\}$ and there is no J-optimal $(25,4,1)$-OOC. By \cite[Theorem 6.12]{CFM}, there exists a J-optimal $(v,4,1)$-OOC for any $v\equiv 0\pmod{24}$.
\end{proof}

We remark that Ge and Yin \cite[Theorem 6.5]{gy} also established the existence of a J-optimal $(v,4,1)$-OOC for any $v\equiv 6 \pmod{12}$ by using different construction methods. For more information on J-optimal $(v,4,1)$-OOC, the reader is referred to \cite{MBuratti2,b18,CM,ChangYin04,KChen,YMiao,fmi,gsm,Yin1}.

\section{Concluding remarks}\label{conclusion}

By utilizing the method presented in \cite{Zhang1}, we established the existence of $(gh,h,4,1)$-CDFs with $h\in\{2,3,6\}$ and $gh\equiv h\pmod{12}$ in Theorem \ref{thm:main}. It shows that the method is flexible and powerful for constructing many kinds of cyclic designs with block size four even though it looks mysterious.

Theorem \ref{thm:RDF-general} provides more new infinite families of $(gh,h,4,1)$-CDFs by using a standard recursive construction. Completing the existence of $(gh,h,4,1)$-CDFs for all admissible parameters $g$ and $h$ will be a complicated and difficult work.

In order to facilitate the reader for further study, we here summarize main known results on the existence of $(gh,h,4,1)$-CDFs not covered by Theorems \ref{thm:cyclic design k=4} and \ref{thm:RDF-general} below.

\begin{Remark}
\begin{enumerate}
\item[$(1)$] {\rm \cite[Theorem 3.9]{CFM}}
There exists a $(gh,h,4,1)$-CDF for $h\in\{18,60,72,96$, $108\}$ and all $g$ such that $\gcd(g,150)=1$ or $25$.
\item[$(2)$] {\rm \cite[Lemmas 3.4, 3.5, 3.10 and 3.11]{CFM}} There exists a $(ph,h,4,1)$-CDF for $h\in\{36,48\}$ and any prime $p>5$.
\item[$(3)$] {\rm \cite[Lemma 6.3]{gy}}
Let $g$ be a positive integer whose prime factors are all congruent to $1$ modulo $4$ and greater than $5$. Then there exists a $(12g,12,4,1)$-CDF.
\item[$(4)$] {\rm \cite[Lemmas 5.2 and 5.3]{gy}}
There exists a $(24\cdot3^{2t}w,24,4,1)$-CDF for any integer $t\geq 0$ and all $w$ such that $\gcd(w,6)=1$.
\item[$(5)$] {\rm \cite[Lemma 4.8]{CM}} There exists a $(2^{s+4},2^s,4,1)$-CDF for any integer $s\geq3$.
\item[$(6)$] {\rm \cite[Lemma 4.4]{ChangYin04}} Let $s\in\{3,4,5,6\}$. There exists a $(2^{4n+s},2^s,4,1)$-CDF for any positive integer $n$.
\item[$(7)$] {\rm \cite[Theorem 4.16]{ChangYin04}}  Let $s\in\{3,4,5,6\}$. There exists a $(2^sp,2^s,4,1)$-CDF for any prime $p\equiv1 \pmod{6}$ except possibly for $(p,s)=(7,6)$.
\end{enumerate}
\end{Remark}

\appendix
\setcounter{equation}{0}
\renewcommand\theequation{A.\arabic{equation}}

\section{The remaining $18+x$ base blocks in Lemma \ref{lem-RCDF-main}}\label{Sec:A}

\begin{center}\tabcolsep 0.03in
\begin{longtable}{|lll|}\hline
$h=2$, $x=0$ and $t$ is odd &&\\\hline
$\{0,1,11t-2,42t\}$&$\{0,2,3t-1,11t+1\}$&
$\{0,\frac{3t-3}{2},23t-1,44t\}$\\$\{0,3t,33t+1,60t+2\}$&
$\{0,3t+1,7t+1,19t+3\}$&$\{0,\frac{9t+1}{2},8t+1,\frac{141t+5}{2}\}$\\
$\{0,5t,21t,28t\}$&$\{0,5t+1,\frac{55t+1}{2},\frac{129t+3}{2}\}$&
$\{0,6t-1,19t+1,35t\}$\\$\{0,7t-2,35t-1,52t\}$&$\{0,8t,33t,58t+2\}$&
$\{0,\frac{19t-1}{2}\frac{83t-1}{2},53t\}$\\
$\{0,11t,34t+1,58t+1\}$&$\{0,12t+1,15t-1,39t\}$&
$\{0,\frac{31t-1}{2}\frac{71t+1}{2},\frac{107t+3}{2}\}$\\
$\{0,16t+1,\frac{35t+3}{2},42t+2\}$&
$\{0,\frac{33t+3}{2},26t,65t+3\}$&$\{0,19t,32t+1,36t+2\}$\\\hline

$h=2$, $x=0$ and $t$ is even &&\\\hline
$\{0,1,8t+1,11t-2\}$&$\{0,2,\frac{19t}{2}+2,28t+2\}$&
$\{0,\frac{3t}{2}+1,3t+1,38t+2\}$\\$\{0,\frac{3t}{2}+2,\frac{25t}{2}+1,\frac{109t}{2}+1\}$&
$\{0,3t-1,29t+1,42t+2\}$&$\{0,4t,21t+1,53t+1\}$\\
$\{0,4t+1,45t+1,51t+2\}$&$\{0,5t,11t,64t\}$&
$\{0,5t+1,16t+2,32t+1\}$\\$\{0,7t-2,35t-1,55t\}$&
$\{0,7t-1,23t,65t+1\}$&$\{0,\frac{19t}{2}+1,33t+2,57t+3\}$\\
$\{0,12t,26t+1,\frac{85t}{2}+2\}$&$\{0,12t+2,15t,54t+1\}$&
$\{0,\frac{31t}{2}-1,23t-1,60t+1\}$\\$\{0,20t+2,36t+2,39t+2\}$&
$\{0,\frac{43t}{2}+1,25t+2,50t+2\}$&$\{0,\frac{45t}{2}-1,27t-1,34t-1\}$\\\hline

$h=2$, $x=1$ and $t$ is odd &&\\\hline
$\{0,1,3t-2,15t\}$&$\{0,2,7t+3,37t+11\}$&
$\{0,\frac{3t-3}{2},\frac{19t+7}{2},\frac{101t+21}{2}\}$\\
$\{0,\frac{3t+1}{2},\frac{25t+3}{2},\frac{141t+29}{2}\}$&
$\{0,3t-1,18t+1,47t+7\}$&$\{0,\frac{7t+1}{2},13t+2,59t+13\}$\\
$\{0,4t,19t+4,52t+10\}$&$\{0,4t+1,25t+5,52t+9\}$&
$\{0,\frac{9t+3}{2},\frac{31t+7}{2},\frac{85t+17}{2}\}$\\
$\{0,5t+2,\frac{55t+9}{2},37t+7\}$&
$\{0,7t,\frac{71t+13}{2},53t+8\}$&$\{0,7t+4,\frac{61t+15}{2},55t+13\}$\\
$\{0,\frac{15t+7}{2},19t+5,61t+14\}$&$\{0,8t+3,32t+6,36t+8\}$&
$\{0,8t+4,35t+6,42t+8\}$ \\
$\{0,12t+1,18t+3,51t+11\}$&
$\{0,14t,19t+3,42t+7\}$&$\{0,15t+3,18t+4,33t+5\}$\\
$\{0,17t+2,23t+3,67t+13\}$&&\\\hline

$h=2$, $x=1$ and $t$ is even &&\\\hline
$\{0,1,13t+2,17t+2\}$&$\{0,\frac{3t}{2},3t+1,18t+1\}$&
$\{0,\frac{3t}{2}+2,\frac{31t}{2}+3,\frac{109t}{2}+13\}$\\$\{0,3t-3,11t+2,64t+10\}$&
$\{0,\frac{7t}{2}+1,48t+11,51t+10\}$&$\{0,4t+1,24t+6,61t+14\}$\\
$\{0,6t+2,24t+4,50t+11\}$&$\{0,6t+3,30t+8,53t+10\}$&
$\{0,7t+1,\frac{71t}{2}+7,45t+12\}$\\$\{0,7t+4,19t+5,37t+9\}$&
$\{0,\frac{15t}{2}+3,20t+4,57t+15\}$&$\{0,8t+3,35t+7,47t+9\}$\\
$\{0,\frac{19t}{2}+4,\frac{47t}{2}+4,\frac{61t}{2}+7\}$&$\{0,15t+2,30t+6,45t+9\}$&
$\{0,18t+3,41t+7,44t+7\}$\\$\{0,\frac{43t}{2}+4,33t+5,40t+7\}$&
$\{0,23t+3,28t+4,28t+6\}$&$\{0,29t+6,32t+4,36t+6\}$\\
$\{0,\frac{59t}{2}+5,34t+6,39t+8\}$ &&\\\hline

$h=2$, $x=2$ and $t$ is odd &&\\\hline
$\{0,1,8t+8,\frac{19t+13}{2}\}$&$\{0,2,3t,35t+10\}$&
$\{0,\frac{3t-1}{2},\frac{19t+9}{2},26t+9\}$\\
$\{0,\frac{3t+1}{2},6t+3,\frac{101t+41}{2}\}$&
$\{0,\frac{7t+5}{2},\frac{61t+23}{2},\frac{87t+31}{2}\}$&$\{0,4t+2,28t+11,56t+21\}$\\
$\{0,4t+3,19t+7,42t+17\}$&
$\{0,5t+2,8t+3,57t+20\}$&$\{0,5t+4,41t+15,44t+17\}$\\$\{0,7t+2,14t+3,22t+5\}$&
$\{0,7t+3,24t+10,37t+15\}$&$\{0,8t+4,25t+10,37t+14\}$\\
$\{0,8t+6,11t+3,23t+8\}$&$\{0,12t+2,17t+5,34t+9\}$&
$\{0,12t+3,30t+10,33t+9\}$\\$\{0,\frac{31t+11}{2},27t+8,58t+22\}$&
$\{0,\frac{35t+13}{2},30t+11,53t+18\}$&$\{0,18t+6,33t+11,36t+14\}$\\
$\{0,\frac{37t+15}{2},26t+10,\frac{99t+39}{2}\}$&$\{0,22t+6,28t+8,47t+17\}$&\\\hline

$h=2$, $x=2$ and $t$ is even &&\\\hline
$\{0,1,8t+8,\frac{19t}{2}+8\}$&$\{0,\frac{3t}{2}+1,6t+3,44t+18\}$&
$\{0,\frac{3t}{2}+2,18t+6,55t+21\}$\\$\{0,3t-3,22t+6,33t+9\}$&
$\{0,3t-2,22t+5,47t+15\}$&$\{0,3t-1,7t+2,15t+4\}$\\
$\{0,3t,27t+9,44t+15\}$&$\{0,3t+1,\frac{31t}{2}+5,39t+15\}$&
$\{0,3t+2,8t+6,20t+9\}$\\$\{0,4t+2,21t+6,36t+11\}$&
$\{0,5t+3,18t+7,42t+17\}$&$\{0,6t+4,\frac{57t}{2}+10,40t+12\}$\\
$\{0,7t+1,25t+9,30t+11\}$&$\{0,\frac{15t}{2}+2,\frac{55t}{2}+8,\frac{71t}{2}+12\}$&
$\{0,8t+3,15t+6,58t+22\}$\\$\{0,8t+5,31t+14,57t+24\}$&
$\{0,\frac{19t}{2}+6,\frac{49t}{2}+9,\frac{101t}{2}+20\}$&$\{0,12t+2,12t+4,49t+20\}$\\
$\{0,13t+5,27t+8,\frac{61t}{2}+11\}$&$\{0,16t+5,19t+8,49t+18\}$&\\\hline

$h=2$, $x=3$ and $t$ is odd &&\\\hline
$\{0,1,\frac{3t-1}{2},64t+29\}$&$\{0,\frac{3t+1}{2},25t+14,56t+29\}$&
$\{0,3t-2,41t+20,60t+31\}$\\$\{0,3t+3,42t+23,58t+31\}$&
$\{0,4t+3,32t+16,69t+39\}$&$\{0,\frac{9t+7}{2},\frac{15t+11}{2},58t+32\}$\\
$\{0,5t+3,34t+17,42t+24\}$&$\{0,5t+4,21t+11,52t+25\}$&
$\{0,6t+3,\frac{85t+49}{2},\frac{99t+55}{2}\}$\\$\{0,6t+4,\frac{19t+13}{2},21t+12\}$&
$\{0,7t+6,24t+15,54t+28\}$&$\{0,8t+5,27t+14,32t+19\}$\\
$\{0,8t+6,12t+8,31t+20\}$&$\{0,8t+8,19t+14,69t+41\}$&
$\{0,\frac{19t+15}{2},25t+15,\frac{111t+59}{2}\}$\\$\{0,12t+5,19t+10,36t+18\}$&
$\{0,\frac{25t+15}{2},\frac{55t+27}{2},\frac{95t+47}{2}\}$&$\{0,13t+7,31t+16,44t+24\}$\\
$\{0,21t+9,33t+15,36t+16\}$&$\{0,23t+11,23t+13,46t+23\}$&
$\{0,29t+13,33t+17,36t+17\}$\\\hline

$h=2$, $x=3$ and $t$ is even &&\\\hline
$\{0,1,8t+8,8t+10\}$&$\{0,\frac{3t}{2}+1,31t+14,\frac{71t}{2}+17\}$& $\{0,\frac{3t}{2}+2,\frac{19t}{2}+8,57t+32\}$\\
$\{0,3t-3,\frac{25t}{2}+7,35t+17\}$&$\{0,3t-1,23t+11,30t+13\}$
&$\{0,3t+1,35t+16,\frac{89t}{2}+25\}$\\
$\{0,3t+2,21t+11,44t+24\}$&$\{0,\frac{7t}{2}+3,15t+8,21t+12\}$&
$\{0,4t+4,19t+9,43t+24\}$\\
$\{0,5t+3,36t+18,61t+32\}$&
$\{0,5t+5,\frac{47t}{2}+14,52t+27\}$&$\{0,6t+5,\frac{15t}{2}+5,\frac{113t}{2}+31\}$\\
$\{0,7t+3,36t+16,53t+24\}$&$\{0,7t+4,19t+11,22t+11\}$&
$\{0,12t+5,16t+8,24t+13\}$\\
$\{0,13t+7,17t+9,36t+21\}$&
$\{0,13t+8,44t+25,47t+23\}$&$\{0,14t+6,35t+15,47t+21\}$\\
$\{0,14t+7,17t+10,56t+31\}$& $\{0,19t+10,24t+14,65t+32\}$&
$\{0,23t+10,30t+15,41t+20\}$\\\hline

$h=2$, $x=4$ and $t$ is odd & & \\\hline
$\{0,\frac{3t-3}{2},14t+10,44t+31\}$&$\{0,\frac{3t+1}{2},23t+16,56t+39\}$&
$\{0,3t-3,3t-1,45t+30\}$\\$\{0,3t-2,8t+5,27t+17\}$&
$\{0,3t,35t+23,68t+45\}$&$\{0,3t+1,\frac{89t+61}{2},54t+38\}$\\
$\{0,4t,23t+14,46t+32\}$&$\{0,4t+1,12t+11,24t+18\}$&
$\{0,4t+3,11t+7,64t+41\}$\\$\{0,4t+4,22t+15,55t+39\}$&
$\{0,6t+6,37t+28,37t+29\}$&$\{0,8t+7,\frac{23t+15}{2},55t+38\}$\\
$\{0,\frac{19t+13}{2},\frac{47t+35}{2},\frac{113t+77}{2}\}$&$\{0,12t+9,27t+19,\frac{73t+55}{2}\}$&
$\{0,12t+10,31t+21,36t+27\}$\\$\{0,13t+12,34t+23,60t+42\}$&
$\{0,15t+9,23t+17,34t+22\}$&$\{0,15t+11,\frac{33t+21}{2},21t+16\}$\\
$\{0,17t+10,32t+22,36t+28\}$&$\{0,20t+12,27t+18,31t+20\}$&
$\{0,\frac{45t+29}{2},30t+20,41t+26\}$\\
$\{0,23t+15,28t+20,36t+26\}$&&\\\hline

$h=2$, $x=4$ and $t$ is even &&\\\hline
$\{0,1,\frac{19t}{2}+10,39t+27\}$&$\{0,2,3t+1,45t+33\}$&
$\{0,\frac{3t}{2},\frac{37t}{2}+11,\frac{97t}{2}+32\}$\\$\{0,3t,\frac{15t}{2}+5,\frac{31t}{2}+11\}$&
$\{0,\frac{7t}{2}+1,\frac{19t}{2}+8,\frac{89t}{2}+31\}$&$\{0,4t,28t+19,36t+24\}$\\
$\{0,4t+1,8t+7,12t+9\}$&$\{0,4t+5,15t+12,68t+46\}$&
$\{0,5t+5,21t+16,57t+39\}$\\$\{0,7t+4,19t+11,58t+39\}$&
$\{0,7t+5,12t+12,61t+44\}$&$\{0,7t+6,31t+23,66t+44\}$\\
$\{0,8t+10,11t+8,49t+36\}$&$\{0,11t+5,30t+19,60t+39\}$&
$\{0,\frac{23t}{2}+7,\frac{49t}{2}+18,55t+38\}$\\$\{0,12t+8,32t+21,47t+30\}$&
$\{0,14t+10,26t+20,47t+31\}$&$\{0,15t+10,19t+13,46t+31\}$\\
$\{0,\frac{33t}{2}+10,18t+12,\frac{109t}{2}+39\}$&$\{0,17t+10,36t+22,41t+28\}$&
$\{0,19t+18,22t+15,49t+35\}$\\$\{0,\frac{43t}{2}+16,23t+17,44t+30\}$&&\\\hline

$h=2$, $x=5$ and $t$ is odd &&\\\hline
$\{0,1,3t-2,15t+12\}$&$\{0,\frac{3t-3}{2},\frac{19t+17}{2},\frac{141t+123}{2}\}$&
$\{0,\frac{3t-1}{2},38t+33,\frac{135t+115}{2}\}$\\$\{0,3t,28t+25,36t+30\}$&
$\{0,3t+1,11t+7,53t+44\}$&$\{0,4t+5,25t+22,39t+34\}$\\
$\{0,5t+4,\frac{55t+47}{2},58t+49\}$&$\{0,5t+5,8t+7,64t+53\}$&
$\{0,7t+4,7t+6,48t+39\}$\\$\{0,7t+5,35t+28,40t+34\}$&
$\{0,\frac{15t+11}{2},\frac{23t+17}{2},\frac{125t+105}{2}\}$&$\{0,\frac{19t+15}{2},\frac{35t+31}{2},66t+57\}$\\
$\{0,12t+8,29t+24,64t+51\}$&$\{0,12t+9,\frac{49t+45}{2},\frac{57t+49}{2}\}$&
$\{0,12t+10,31t+30,53t+47\}$\\$\{0,12t+11,25t+24,46t+39\}$&
$\{0,12t+12,28t+24,48t+40\}$&$\{0,15t+13,31t+28,42t+36\}$\\
$\{0,16t+13,19t+16,43t+37\}$&$\{0,16t+14,39t+35,46t+38\}$&
$\{0,18t+16,22t+20,49t+42\}$\\$\{0,\frac{37t+31}{2},22t+18,40t+33\}$&
$\{0,24t+20,36t+33,47t+39\}$ &\\\hline

$h=2$, $x=5$ and $t$ is even &&\\\hline
$\{0,1,3t-2,3t\}$&$\{0,3t+1,23t+19,56t+47\}$&
$\{0,\frac{7t}{2}+3,\frac{37t}{2}+15,\frac{85t}{2}+38\}$\\$\{0,4t+2,8t+5,64t+51\}$&
$\{0,4t+5,12t+12,31t+29\}$&$\{0,5t+5,12t+10,50t+45\}$\\
$\{0,6t+5,28t+24,59t+49\}$&$\{0,7t+4,11t+8,43t+38\}$&
$\{0,\frac{15t}{2}+5,\frac{61t}{2}+25,\frac{97t}{2}+41\}$\\$\{0,8t+6,41t+35,58t+50\}$&
$\{0,8t+8,19t+15,56t+50\}$&$\{0,11t+6,41t+32,47t+38\}$\\
$\{0,11t+9,37t+34,58t+49\}$&$\{0,12t+9,29t+25,56t+48\}$&
$\{0,12t+14,19t+20,60t+54\}$\\$\{0,13t+14,\frac{49t}{2}+22,60t+51\}$&
$\{0,\frac{31t}{2}+12,25t+23,28t+25\}$&$\{0,16t+13,\frac{35t}{2}+15,67t+58\}$\\
$\{0,\frac{33t}{2}+15,18t+15,26t+24\}$&$\{0,21t+18,24t+21,57t+51\}$&
$\{0,22t+20,42t+37,\frac{87t}{2}+38\}$\\$\{0,23t+21,30t+24,38t+34\}$&
$\{0,\frac{55t}{2}+23,32t+27,37t+33\}$&\\\hline
&&\\\hline

$h=3$, $x=0$ and $t$ is odd &&\\\hline
$\{0,1,3t-2,11t-2\}$&$\{0,2,7t+1,19t+4\}$&
$\{0,\frac{3t-1}{2},32t+2,\frac{71t+3}{2}\}$\\$\{0,3t-1,16t+2,23t+2\}$&
$\{0,4t+1,21t+2,44t+3\}$&$\{0,4t+2,8t+1,46t+3\}$\\
$\{0,\frac{9t+3}{2},\frac{33t+3}{2},\frac{87t+3}{2}\}$&$\{0,5t+1,19t+3,30t+3\}$&	$\{0,5t+2,\frac{25t+5}{2},42t+3\}$\\$\{0,6t,18t+2,37t+3\}$&$\{0,6t+1,\frac{35t+3}{2},\frac{125t+7}{2}\}$&
$\{0,8t+2,\frac{19t+1}{2},31t+2\}$\\
$\{0,\frac{19t-3}{2},33t,59t+1\}$&$\{0,11t-1,28t+1,56t+3\}$&
$\{0,11t+1,33t+2,47t+3\}$\\
$\{0,12t+1,16t+1,37t+2\}$&
$\{0,15t+1,18t+1,57t+3\}$&$\{0,\frac{37t+3}{2},20t+2,\frac{95t+5}{2}\}$\\\hline
	
$h=3$, $x=0$ and $t$ is even &&\\\hline
	$\{0,1,8t+1,11t-2\}$&$\{0,2,7t+1,19t+4\}$&
	$\{0,\frac{3t}{2},32t+2,44t+3\}$,\\$\{0,\frac{3t}{2}+2,19t+3,30t+3\}$&
	$\{0,3t-2,7t,59t\}$&$\{0,3t-1,11t+1,36t+2\}$\\
	$\{0,3t,\frac{55t}{2},\frac{141t}{2}+2\}$,&$\{0,4t+1,18t+2,45t+2\}$&
	$\{0,5t+1,33t+2,37t+2\}$\\$\{0,5t+2,30t+2,56t+3\}$&
	$\{0,6t+1,28t+2,45t+4\}$&$\{0,6t+2,\frac{19t}{2}+2,\frac{113t}{2}+3\}$\\
	$\{0,\frac{15t}{2},\frac{45t}{2},\frac{119t}{2}+1\}$&$\{0,\frac{19t}{2}+1,19t+1,\frac{111t}{2}+2\}$&
	$\{0,11t-1,34t+1,49t+2\}$,\\$\{0,\frac{23t}{2},30t+1,60t+1\}$&
	$\{0,12t,35t,51t+2\}$&$\{0,14t+2,18t+1,31t+3\}$\\\hline

$h=3$, $x=1$ and $t$ is odd &&\\\hline
$\{0,1,3t-2,11t+3\}$&$\{0,\frac{3t-3}{2},\frac{45t+5}{2},\frac{119t+23}{2}\}$&
$\{0,\frac{3t-1}{2},\frac{19t+3}{2},\frac{49t+9}{2}\}$\\
$\{0,\frac{3t+1}{2},\frac{55t+13}{2},\frac{61t+13}{2}\}$&
$\{0,3t-1,7t,52t+10\}$&$\{0,4t,\frac{43t+7}{2},31t+7\}$\\
$\{0,4t+2,27t+4,39t+7\}$&$\{0,\frac{9t+3}{2},\frac{23t+7}{2},\frac{73t+17}{2}\}$&
$\{0,5t+1,41t+7,41t+9\}$\\
$\{0,6t+1,22t+3,67t+12\}$&
$\{0,7t+1,13t+3,30t+7\}$&$\{0,\frac{15t+3}{2},\frac{31t+5}{2},\frac{47t+11}{2}\}$\\
$\{0,8t+4,19t+4,38t+9\}$&$\{0,13t+4,\frac{33t+9}{2},46t+10\}$&
$\{0,14t+3,25t+4,30t+6\}$ \\
$\{0,15t+1,33t+5,53t+9\}$&
$\{0,16t+5,37t+8,49t+12\}$&$\{0,19t+3,\frac{57t+11}{2},47t+9\}$\\
$\{0,29t+5,32t+6,47t+8\}$&&\\\hline

$h=3$, $x=1$ and $t$ is even &&\\\hline
$\{0,2,\frac{3t}{2}+2,\frac{19t}{2}+5\}$&$\{0,\frac{3t}{2}+1,43t+10,49t+12\}$&
$\{0,3t-3,23t+2,60t+11\}$\\$\{0,4t+1,16t+4,51t+11\}$&
$\{0,5t+1,46t+8,54t+12\}$&$\{0,6t+3,\frac{19t}{2}+4,\frac{99t}{2}+13\}$\\
$\{0,7t+2,\frac{47t}{2}+6,53t+11\}$&$\{0,\frac{15t}{2}+1,25t+4,64t+14\}$&
$\{0,8t+2,13t+4,65t+15\}$\\$\{0,11t,19t+5,30t+6\}$&
$\{0,11t+2,27t+4,67t+12\}$&$\{0,11t+3,42t+10,45t+8\}$\\
$\{0,11t+4,30t+7,45t+10\}$&$\{0,\frac{23t}{2}+3,33t+7,56t+12\}$&
$\{0,\frac{37t}{2}+3,31t+6,\frac{71t}{2}+7\}$\\$\{0,21t+3,35t+5,39t+7\}$&
$\{0,22t+3,36t+6,49t+9\}$&$\{0,25t+6,25t+7,28t+6\}$\\
$\{0,29t+6,33t+6,36t+7\}$ &&\\\hline

$h=3$, $x=2$ and $t$ is odd &&\\\hline
$\{0,1,8t+8,\frac{19t+13}{2}\}$&$\{0,\frac{3t-1}{2},18t+6,38t+15\}$&
$\{0,3t-3,3t-1,60t+20\}$\\
$\{0,3t-2,35t+11,46t+16\}$&
$\{0,3t,26t+10,\frac{95t+35}{2}\}$&$\{0,3t+2,\frac{25t+13}{2},42t+16\}$\\
$\{0,3t+3,8t+6,19t+6\}$&
$\{0,\frac{7t+5}{2},\frac{37t+15}{2},\frac{87t+35}{2}\}$&$\{0,4t+2,11t+3,53t+20\}$\\$\{0,6t+3,25t+12,37t+17\}$&
$\{0,7t+3,21t+8,25t+11\}$&$\{0,12t+4,43t+17,51t+20\}$\\
$\{0,15t+3,23t+7,59t+21\}$&$\{0,16t+5,22t+7,58t+23\}$&
$\{0,16t+6,\frac{35t+13}{2},33t+12\}$\\$\{0,19t+8,24t+10,55t+20\}$&
$\{0,\frac{45t+13}{2},27t+9,34t+11\}$&$\{0,23t+6,28t+10,31t+11\}$\\
$\{0,23t+8,\frac{61t+21}{2},54t+20\}$&$\{0,25t+9,37t+15,45t+17\}$&\\\hline

$h=3$, $x=2$ and $t$ is even &&\\\hline
	$\{0,1,8t+6,8t+8\}$&$\{0,\frac{3t}{2},\frac{35t}{2}+6,\frac{101t}{2}+19\}$&
	$\{0,\frac{3t}{2}+1,\frac{37t}{2}+7,\frac{141t}{2}+25\}$\\$\{0,3t-3,29t+9,41t+14\}$&
	$\{0,3t-2,\frac{25t}{2}+6,35t+12\}$&$\{0,3t,19t+7,68t+24\}$\\
	$\{0,3t+1,\frac{61t}{2}+10,67t+25\}$&$\{0,\frac{7t}{2}+3,\frac{19t}{2}+7,\frac{113t}{2}+22\}$&
	$\{0,\frac{9t}{2}+2,\frac{19t}{2}+6,\frac{85t}{2}+18\}$\\$\{0,5t+3,8t+2,51t+19\}$&
	$\{0,8t+4,30t+11,49t+21\}$&$\{0,11t+3,36t+14,42t+17\}$\\
	$\{0,11t+4,27t+9,68t+25\}$&$\{0,11t+5,23t+9,36t+15\}$&
	$\{0,\frac{23t}{2}+4,\frac{47t}{2}+10,\frac{111t}{2}+21\}$\\$\{0,12t+7,27t+10,47t+18\}$&
	$\{0,14t+6,21t+7,58t+24\}$&$\{0,15t+4,18t+6,33t+11\}$\\
	$\{0,17t+7,\frac{49t}{2}+9,53t+18\}$&$\{0,18t+8,41t+15,48t+17\}$&\\\hline

$h=3$, $x=3$ and $t$ is odd &&\\\hline
$\{0,1,8t+10,\frac{19t+17}{2}\}$&$\{0,2,31t+15,64t+33\}$&
$\{0,\frac{3t-1}{2},\frac{33t+17}{2},\frac{113t+59}{2}\}$\\$\{0,\frac{3t+1}{2},23t+10,30t+14\}$&
$\{0,3t,23t+11,\frac{55t+29}{2}\}$&$\{0,3t+1,21t+10,26t+15\}$\\
$\{0,3t+3,36t+22,41t+25\}$&$\{0,\frac{7t+5}{2},11t+6,66t+35\}$&
$\{0,5t+4,13t+7,\frac{71t+35}{2}\}$\\$\{0,6t+3,14t+7,57t+33\}$&
$\{0,7t+2,32t+19,35t+16\}$&$\{0,3t-2,34t+18,53t+27\}$\\
$\{0,8t+5,11t+7,31t+17\}$&$\{0,11t+5,\frac{59t+29}{2},47t+24\}$&
$\{0,\frac{23t+13}{2},\frac{47t+27}{2},\frac{119t+63}{2}\}$\\$\{0,12t+6,15t+5,41t+21\}$&
$\{0,15t+8,19t+10,68t+36\}$&$\{0,16t+9,31t+16,55t+31\}$\\
$\{0,22t+9,35t+17,47t+25\}$&$\{0,22t+11,29t+14,37t+21\}$&
$\{0,27t+15,31t+19,45t+25\}$\\\hline

$h=3$, $x=3$ and $t$ is even &&\\\hline
$\{0,1,8t+10,\frac{19t}{2}+10\}$&$\{0,3t,27t+14,68t+36\}$& $\{0,3t+1,21t+10,52t+29\}$\\
$\{0,3t+3,8t+7,39t+20\}$&$\{0,4t+2,15t+8,72t+37\}$
&$\{0,5t+3,17t+9,36t+18\}$\\
$\{0,5t+5,19t+11,69t+41\}$&$\{0,7t+4,32t+20,\frac{99t}{2}+29\}$&
$\{0,\frac{15t}{2}+3,23t+12,49t+29\}$\\
$\{0,8t+6,\frac{49t}{2}+14,36t+20\}$&
$\{0,\frac{19t}{2}+8,\frac{89t}{2}+25,\frac{101t}{2}+29\}$&$\{0,11t+5,\frac{25t}{2}+7,17t+10\}$\\
$\{0,12t+8,15t+5,53t+27\}$&$\{0,13t+8,35t+18,60t+32\}$&
$\{0,14t+7,22t+11,29t+13\}$\\
$\{0,15t+7,33t+17,56t+30\}$&
$\{0,15t+9,38t+23,41t+25\}$&$\{0,20t+11,\frac{47t}{2}+14,59t+32\}$\\
$\{0,21t+9,29t+14,48t+24\}$& $\{0,\frac{57t}{2}+13,30t+14,47t+22\}$&
$\{0,29t+15,37t+23,40t+22\}$\\\hline

$h=3$, $x=4$ and $t$ is odd & & \\\hline
$\{0,\frac{3t-3}{2},3t-2,11t+7\}$&$\{0,3t+1,36t+24,39t+26\}$&
$\{0,\frac{7t+5}{2},\frac{83t+59}{2},\frac{89t+65}{2}\}$\\$\{0,4t+4,11t+6,29t+20\}$&
$\{0,6t+4,\frac{37t+31}{2},43t+32\}$&$\{0,7t+4,33t+21,60t+39\}$\\
$\{0,7t+5,47t+34,51t+36\}$&$\{0,7t+6,12t+9,67t+46\}$&
$\{0,\frac{15t+11}{2},25t+19,57t+42\}$\\$\{0,8t+5,25t+18,45t+34\}$&
$\{0,8t+8,\frac{73t+55}{2},38t+28\}$&$\{0,\frac{19t+13}{2},14t+10,31t+22\}$\\
$\{0,\frac{19t+15}{2},21t+16,65t+48\}$&$\{0,11t+5,30t+22,56t+40\}$&
$\{0,11t+8,19t+18,55t+41\}$\\$\{0,12t+8,30t+21,53t+36\}$&
$\{0,12t+10,15t+10,29t+21\}$&$\{0,13t+11,18t+15,72t+50\}$\\
$\{0,15t+11,23t+18,64t+45\}$&$\{0,\frac{31t+23}{2},\frac{85t+61}{2},\frac{97t+67}{2}\}$&
$\{0,20t+17,23t+14,23t+16\}$\\
$\{0,22t+15,33t+24,37t+27\}$&&\\\hline

$h=3$, $x=4$ and $t$ is even &&\\\hline
$\{0,1,3t-2,53t+34\}$&$\{0,2,\frac{19t}{2}+10,69t+50\}$&
$\{0,\frac{3t}{2},\frac{19t}{2}+9,\frac{137t}{2}+48\}$\\$\{0,3t-1,11t+9,64t+44\}$&
$\{0,3t+2,21t+15,60t+42\}$&$\{0,\frac{9t}{2}+3,41t+30,49t+36\}$\\
$\{0,5t+4,31t+23,36t+28\}$&$\{0,7t+2,15t+10,49t+33\}$&
$\{0,7t+3,26t+18,54t+36\}$\\$\{0,7t+5,30t+22,45t+34\}$&
$\{0,7t+6,11t+8,43t+32\}$&$\{0,\frac{15t}{2}+5,\frac{87t}{2}+32,55t+40\}$\\
$\{0,8t+5,35t+23,55t+39\}$&$\{0,11t+5,28t+19,60t+40\}$&
$\{0,12t+8,35t+22,47t+32\}$\\$\{0,13t+11,16t+11,55t+41\}$&
$\{0,15t+9,\frac{33t}{2}+10,58t+40\}$&$\{0,17t+13,21t+16,28t+20\}$\\
$\{0,\frac{35t}{2}+13,\frac{47t}{2}+18,39t+29\}$&$\{0,\frac{45t}{2}+14,41t+29,\frac{85t}{2}+31\}$&
$\{0,25t+16,39t+26,43t+30\}$\\$\{0,27t+19,33t+23,45t+35\}$&&\\\hline

$h=3$, $x=5$ and $t$ is odd &&\\\hline
$\{0,1,\frac{3t-1}{2},45t+40\}$&$\{0,\frac{3t+1}{2},\frac{33t+29}{2},\frac{125t+111}{2}\}$&
$\{0,3t-1,35t+29,39t+32\}$\\$\{0,3t+1,21t+16,64t+54\}$&
$\{0,3t+3,41t+35,57t+50\}$&$\{0,\frac{7t+5}{2},\frac{55t+49}{2},58t+50\}$\\
$\{0,4t+4,22t+21,61t+52\}$&$\{0,4t+5,12t+13,15t+11\}$&
$\{0,5t+5,8t+5,8t+7\}$\\$\{0,5t+6,29t+26,41t+38\}$&
$\{0,6t+4,39t+34,56t+49\}$&$\{0,\frac{15t+9}{2},\frac{43t+33}{2},\frac{125t+107}{2}\}$\\
$\{0,8t+6,30t+26,49t+42\}$&$\{0,8t+10,19t+19,69t+66\}$&
$\{0,\frac{23t+21}{2},27t+25,\frac{73t+67}{2}\}$\\$\{0,12t+10,41t+32,49t+43\}$&
$\{0,12t+11,34t+30,41t+34\}$&$\{0,12t+14,19t+17,39t+35\}$\\
$\{0,\frac{25t+27}{2},31t+30,36t+34\}$&$\{0,13t+13,21t+17,66t+58\}$&
$\{0,13t+14,17t+16,43t+39\}$\\$\{0,25t+21,\frac{59t+51}{2},\frac{99t+85}{2}\}$&
$\{0,25t+24,32t+29,47t+41\}$ &\\\hline

$h=3$, $x=5$ and $t$ is even &&\\\hline
$\{0,1,33t+29,41t+33\}$&$\{0,\frac{3t}{2},11t+11,45t+41\}$&
$\{0,\frac{3t}{2}+1,3t+3,50t+43\}$\\$\{0,3t-2,11t+9,15t+12\}$&
$\{0,3t+1,35t+29,\frac{119t}{2}+50\}$&$\{0,3t+2,39t+32,52t+45\}$\\
$\{0,4t+4,31t+29,34t+29\}$&$\{0,4t+5,15t+15,64t+57\}$&
$\{0,\frac{9t}{2}+4,\frac{19t}{2}+9,\frac{109t}{2}+48\}$\\$\{0,6t+5,14t+12,57t+52\}$&
$\{0,7t+4,26t+23,29t+22\}$&$\{0,7t+5,43t+37,48t+43\}$\\
$\{0,\frac{15t}{2}+4,17t+14,58t+50\}$&$\{0,8t+5,20t+17,25t+21\}$&
$\{0,8t+10,11t+7,47t+41\}$\\$\{0,\frac{23t}{2}+10,\frac{55t}{2}+24,\frac{85t}{2}+38\}$&
$\{0,13t+14,32t+31,50t+47\}$&$\{0,15t+13,\frac{37t}{2}+16,54t+46\}$\\
$\{0,16t+15,27t+21,27t+23\}$&$\{0,17t+15,21t+17,29t+25\}$&
$\{0,19t+18,25t+24,\frac{83t}{2}+38\}$\\$\{0,22t+17,30t+26,48t+41\}$&
$\{0,22t+21,29t+24,41t+37\}$&\\\hline
& & \\\hline

$h=6$, $x=0$ and $t$ is odd &&\\\hline
			$\{0,1,8t+1,11t-2\}$&$\{0,2,\frac{19t+1}{2},\frac{141t+15}{2}\}$&
			$\{0,\frac{3t-1}{2},\frac{33t+3}{2},44t+4\}$\\$\{0,3t-2,7t-2,27t-1\}$&
			$\{0,3t,\frac{43t+1}{2},\frac{119t+9}{2}\}$&$\{0,3t+1,22t+2,33t+3\}$\\
			$\{0,\frac{7t+1}{2},\frac{57t+5}{2},\frac{73t+9}{2}\}$&$\{0,4t+1,\frac{23t+3}{2},59t+5\}$&
			$\{0,5t+1,39t+4,46t+3\}$\\$\{0,6t-1,\frac{47t+1}{2},54t+5\}$&$\{0,7t,19t+2,49t+5\}$&$\{0,7t+1,37t+5,57t+7\}$\\
			$\{0,\frac{19t-1}{2},14t,59t+4\}$&$\{0,11t+2,28t+3,43t+4\}$&
			$\{0,14t+1,39t+5,50t+5\}$\\
			$\{0,16t+1,19t,56t+4\}$&
			$\{0,21t+1,27t+1,57t+6\}$&$\{0,21t,26t,45t+3\}$
			\\\hline

$h=6$, $x=0$ and $t$ is even &&\\\hline
			$\{0,1,3t-2,11t-2\}$&$\{0,2,27t+1,42t+3\}$&
			$\{0,\frac{3t}{2},30t+2,69t+6\}$,\\$\{0,\frac{3t}{2}+1,\frac{19t}{2}+2,57t+6\}$&
			$\{0,\frac{3t}{2}+2,11t+2,24t+3\}$&$\{0,\frac{7t}{2}+1,\frac{37t}{2},\frac{113t}{2}+5\}$\\
			$\{0,4t,7t-1,56t+4\}$&$\{0,4t+1,11t-1,39t+2\}$&
			$\{0,5t,24t,50t+4\}$\\$\{0,5t+1,29t+2,69t+5\}$&
			$\{0,6t,37t+4,\frac{97t}{2}+5\}$&$\{0,7t+1,28t+2,60t+4\}$\\
			$\{0,\frac{15t}{2},20t+1,\frac{83t}{2}+2\}$&$\{0,11,30t+1,66t+5\}$&
			$\{0,11t+1,25t+2,50t+6\}$\\$\{0,14t,21t,36t+1\}$&
			$\{0,16t+1,33t+3,46t+5\}$&$\{0,18t+1,\frac{55t}{2}+2,\frac{109t}{2}+5\}$\\\hline

$h=6$, $x=1$ and $t$ is odd & & \\\hline
			$\{0,1,5t+2,54t+15\}$&$\{0,\frac{3t-3}{2},\frac{19t+3}{2},\frac{83t+17}{2}\}$&
			$\{0,\frac{3t-1}{2},\frac{43t+7}{2},\frac{121t+29}{2}\}$\\
			$\{0,\frac{3t+1}{2},\frac{47t+11}{2},51t+14\}$&
			$\{0,3t-2,8t+1,53t+12\}$&$\{0,3t+1,39t+12,42t+11\}$\\
			$\{0,\frac{7t+1}{2},16t+4,68t+17\}$&$\{0,\frac{9t+3}{2},\frac{85t+25}{2},58t+16\}$&
			$\{0,6t+1,14t+3,51t+15\}$\\
			$\{0,7t,\frac{71t+13}{2},\frac{125t+29}{2}\}$&
			$\{0,7t+1,26t+5,56t+15\}$&$\{0,7t+2,\frac{49t+11}{2},34t+8\}$\\
			$\{0,8t+5,30t+8,43t+12\}$&$\{0,11t,11t+2,15t+2\}$&
			$\{0,11t+4,30t+9,47t+12\}$ \\
			$\{0,12t+4,15t+4,57t+17\}$&
			$\{0,15t+3,26t+6,45t+9\}$&$\{0,17t+4,23t+6,48t+13\}$\\
			$\{0,25t+5,33t+9,36t+6\}$&&\\\hline

$h=6$, $x=1$ and $t$ is even &&\\\hline
			$\{0,1,\frac{19t}{2}+5,39t+10\}$&$\{0,\frac{3t}{2},\frac{3t}{2}+2,\frac{25t}{2}+3\}$&
			$\{0,\frac{3t}{2}+1,11t+4,46t+11\}$\\$\{0,3t,21t+3,36t+7\}$&
			$\{0,4t,49t+12,54t+14\}$&$\{0,6t+2,45t+14,48t+13\}$\\
			$\{0,\frac{15t}{2}+1,11t+2,58t+15\}$&$\{0,8t+1,16t+4,30t+6\}$&
			$\{0,8t+2,23t+5,42t+8\}$\\$\{0,11t,23t+4,68t+17\}$&
			$\{0,13t+4,30t+8,45t+10\}$&$\{0,\frac{31t}{2}+3,\frac{45t}{2}+4,\frac{83t}{2}+9\}$\\
			$\{0,16t+3,21t+4,41t+9\}$&$\{0,17t+3,20t+4,25t+7\}$&
			$\{0,\frac{37t}{2}+3,36t+6,42t+9\}$\\$\{0,19t+4,27t+9,\frac{111t}{2}+15\}$&
			$\{0,19t+6,22t+3,59t+15\}$&$\{0,\frac{49t}{2}+5,29t+6,36t+8\}$\\
			$\{0,27t+7,30t+5,34t+7\}$ &&\\\hline

$h=6$, $x=2$ and $t$ is odd &&\\\hline
			$\{0,1,3t-2,64t+23\}$&$\{0,2,35t+15,\frac{125t+51}{2}\}$&
			$\{0,\frac{3t-3}{2},3t-1,32t+12\}$\\
			$\{0,3t+2,8t+6,19t+9\}$&
			$\{0,3t+3,17t+6,56t+25\}$&$\{0,4t+2,\frac{33t+13}{2},18t+6\}$\\
			$\{0,\frac{9t+5}{2},8t+5,23t+10\}$&
			$\{0,5t+2,16t+6,33t+13\}$&$\{0,5t+3,28t+12,39t+18\}$\\$\{0,7t+2,43t+18,46t+18\}$&
			$\{0,\frac{15t+9}{2},\frac{57t+23}{2},\frac{97t+41}{2}\}$&$\{0,\frac{19t+11}{2},\frac{49t+23}{2},\frac{101t+45}{2}\}$\\
			$\{0,\frac{23t+11}{2},27t+11,34t+12\}$&$\{0,13t+4,20t+8,35t+11\}$&
			$\{0,13t+5,37t+16,43t+19\}$\\$\{0,18t+7,21t+8,66t+28\}$&
			$\{0,18t+8,25t+11,\frac{109t+47}{2}\}$&$\{0,23t+7,30t+12,53t+20\}$\\
			$\{0,25t+12,33t+16,48t+18\}$&$\{0,28t+10,32t+13,47t+17\}$&\\\hline

$h=6$, $x=2$ and $t$ is even &&\\\hline
			$\{0,1,15t+3,23t+7\}$&$\{0,2,3t-1,64t+24\}$&
			$\{0,\frac{3t}{2},\frac{33t}{2}+6,\frac{73t}{2}+15\}$\\$\{0,\frac{3t}{2}+1,\frac{19t}{2}+8,32t+14\}$&
			$\{0,3t-2,22t+7,28t+11\}$&$\{0,3t,33t+12,36t+14\}$\\
			$\{0,3t+1,19t+8,67t+26\}$&$\{0,3t+3,\frac{49t}{2}+11,26t+13\}$&
			$\{0,\frac{7t}{2}+3,\frac{19t}{2}+6,\frac{89t}{2}+20\}$\\$\{0,4t+2,11t+6,49t+22\}$&
			$\{0,4t+3,17t+7,57t+26\}$&$\{0,5t+2,29t+13,56t+24\}$\\
			$\{0,5t+3,23t+9,39t+14\}$&$\{0,7t+1,14t+3,33t+13\}$&
			$\{0,7t+3,20t+8,35t+13\}$\\$\{0,7t+5,28t+12,42t+16\}$&
			$\{0,\frac{15t}{2}+4,\frac{37t}{2}+7,\frac{97t}{2}+20\}$&$\{0,\frac{19t}{2}+7,25t+12,\frac{85t}{2}+18\}$\\
			$\{0,11t+4,29t+11,54t+22\}$&$\{0,\frac{23t}{2}+5,24t+9,\frac{57t}{2}+11\}$&\\\hline

$h=6$, $x=3$ and $t$ is odd &&\\\hline
			$\{0,1,23t+11,64t+33\}$&$\{0,\frac{3t-3}{2},11t+5,11t+7\}$&
			$\{0,\frac{3t-1}{2},\frac{37t+19}{2},\frac{89t+51}{2}\}$\\$\{0,\frac{3t+1}{2},6t+4,30t+17\}$&
			$\{0,3t-3,15t+5,53t+30\}$&$\{0,3t-2,28t+16,36t+19\}$\\
			$\{0,3t+1,16t+8,43t+25\}$&$\{0,4t+3,31t+19,39t+23\}$&
			$\{0,4t+4,19t+10,69t+42\}$\\$\{0,5t+3,23t+13,35t+19\}$&
			$\{0,5t+5,25t+15,49t+30\}$&$\{0,6t+3,\frac{35t+19}{2},\frac{85t+51}{2}\}$\\
			$\{0,7t+2,15t+8,\frac{113t+67}{2}\}$&$\{0,\frac{15t+7}{2},20t+11,\frac{47t+27}{2}\}$&
			$\{0,8t+5,23t+14,61t+36\}$\\$\{0,8t+8,25t+17,41t+26\}$&
			$\{0,\frac{19t+15}{2},26t+17,\frac{101t+65}{2}\}$&$\{0,13t+8,37t+24,40t+23\}$\\
			$\{0,14t+6,21t+9,35t+16\}$&$\{0,15t+7,18t+9,68t+40\}$&
			$\{0,31t+18,36t+22,39t+25\}$\\\hline

$h=6$, $x=3$ and $t$ is even &&\\\hline
			$\{0,1,41t+23,64t+33\}$&$\{0,2,3t-1,22+11\}$& $\{0,\frac{3t}{2},\frac{87t}{2}+26,56t+33\}$\\
			$\{0,\frac{3t}{2}+1,\frac{45t}{2}+10,\frac{121t}{2}+36\}$&$\{0,3t+1,22t+10,35t+18\}$
			&$\{0,3t+2,7t+4,36t+20\}$\\
			$\{0,4t+3,7t+3,68t+38\}$&$\{0,5t+3,19t+10,35t+20\}$&
			$\{0,5t+4,12t+6,20t+10\}$\\
			$\{0,5t+5,19t+11,\frac{47t}{2}+14\}$&
			$\{0,6t+4,24t+13,49t+29\}$&$\{0,6t+5,24t+15,51t+32\}$\\
			$\{0,8t+3,23t+11,35t+19\}$&$\{0,8t+5,23t+12,55t+32\}$&
			$\{0,8t+6,\frac{59t}{2}+16,\frac{141t}{2}+40\}$\\
			$\{0,8t+8,24t+16,41t+25\}$&
			$\{0,11t+5,36t+23,39t+21\}$&$\{0,15t+5,\frac{49t}{2}+15,52t+31\}$\\
			$\{0,\frac{31t}{2}+8,25t+17,\frac{83t}{2}+26\}$& $\{0,23t+14,26t+17,57t+33\}$&
			$\{0,28t+17,\frac{71t}{2}+20,39t+23\}$\\\hline

$h=6$, $x=4$ and $t$ is odd & & \\\hline
			$\{0,1,34t+24,\frac{141t+109}{2}\}$&$\{0,2,37t+32,65t+51\}$&
			$\{0,\frac{3t-3}{2},\frac{87t+63}{2},\frac{119t+85}{2}\}$\\$\{0,3t-3,28t+18,64t+44\}$&
			$\{0,3t-2,42t+29,60t+42\}$&$\{0,3t-1,15t+9,23t+16\}$\\
			$\{0,3t,11t+5,\frac{107t+77}{2}\}$&$\{0,3t+2,11t+8,28t+22\}$&
			$\{0,4t+3,7t+6,55t+41\}$\\$\{0,\frac{9t+7}{2},14t+11,\frac{89t+71}{2}\}$&
			$\{0,5t+4,22t+15,61t+48\}$&$\{0,6t+4,\frac{47t+35}{2},\frac{125t+95}{2}\}$\\
			$\{0,7t+2,26t+19,56t+42\}$&$\{0,7t+4,26t+20,30t+22\}$&
			$\{0,\frac{15t+11}{2},24t+17,30t+20\}$\\$\{0,8t+8,36t+31,39t+32\}$&
			$\{0,\frac{19t+17}{2},13t+11,\frac{49t+39}{2}\}$&$\{0,11t+9,19t+18,55t+42\}$\\
			$\{0,12t+11,39t+29,44t+34\}$&$\{0,13t+12,18t+15,36t+29\}$&
			$\{0,19t+15,37t+31,49t+39\}$\\
			$\{0,20t+16,24t+20,45t+35\}$&&\\\hline

$h=6$, $x=4$ and $t$ is even &&\\\hline
			$\{0,1,34t+23,\frac{125t}{2}+45\}$&$\{0,\frac{3t}{2}+2,18t+13,\frac{137t}{2}+51\}$&
			$\{0,3t-3,23t+14,42t+30\}$\\$\{0,3t-1,28t+18,36t+23\}$&
			$\{0,3t,21t+15,54t+38\}$&$\{0,3t+2,\frac{15t}{2}+5,35t+23\}$\\
			$\{0,4t+2,11t+6,\frac{121t}{2}+46\}$&$\{0,4t+3,18t+14,48t+34\}$&
			$\{0,4t+4,32t+24,69t+56\}$\\$\{0,5t+4,8t+7,25t+20\}$&
			$\{0,6t+4,\frac{37t}{2}+15,\frac{85t}{2}+34\}$&$\{0,6t+5,\frac{47t}{2}+18,\frac{61t}{2}+24\}$\\
			$\{0,7t+5,30t+21,35t+24\}$&$\{0,8t+8,\frac{19t}{2}+8,11t+9\}$&
			$\{0,8t+10,19t+18,55t+43\}$\\$\{0,12t+8,27t+18,57t+43\}$&
			$\{0,12t+10,12t+12,23t+17\}$&$\{0,13t+11,39t+32,55t+42\}$\\
			$\{0,14t+10,31t+24,53t+39\}$&$\{0,15t+9,23t+18,59t+42\}$&
			$\{0,16t+11,28t+22,64t+48\}$\\$\{0,16t+12,23t+15,42t+32\}$&&\\\hline

$h=6$, $x=5$ and $t$ is odd &&\\\hline
			$\{0,1,\frac{3t-1}{2},52t+49\}$&$\{0,2,4t+4,22t+20\}$&
			$\{0,\frac{3t+1}{2},\frac{19t+19}{2},26t+25\}$\\$\{0,3t,42t+38,48t+43\}$&
			$\{0,3t+1,8t+6,19t+16\}$&$\{0,3t+3,27t+24,51t+49\}$\\
			$\{0,\frac{7t+5}{2},11t+7,53t+48\}$&$\{0,\frac{9t+9}{2},\frac{61t+57}{2},\frac{125t+117}{2}\}$&
			$\{0,5t+4,36t+35,39t+34\}$\\$\{0,5t+6,8t+4,69t+64\}$&
			$\{0,6t+4,42t+40,56t+53\}$&$\{0,7t+4,35t+29,60t+53\}$\\
			$\{0,8t+10,29t+26,65t+63\}$&$\{0,8t+11,16t+16,\frac{87t+81}{2}\}$&
			$\{0,11t+8,42t+37,55t+50\}$\\$\{0,12t+10,45t+41,49t+46\}$&
			$\{0,12t+12,30t+27,59t+52\}$&$\{0,12t+14,15t+11,65t+61\}$\\
			$\{0,\frac{25t+27}{2},24t+24,\frac{107t+99}{2}\}$&$\{0,\frac{31t+27}{2},\frac{47t+41}{2},33t+29\}$&
			$\{0,16t+14,31t+28,54t+49\}$\\$\{0,17t+15,28t+26,50t+45\}$&
			$\{0,25t+25,36t+34,40t+37\}$ &\\\hline

$h=6$, $x=5$ and $t$ is even &&\\\hline
			$\{0,\frac{3t}{2},24t+20,30t+25\}$&$\{0,3t-3,3t-1,53t+47\}$&
			$\{0,3t-2,34t+29,64t+55\}$\\$\{0,3t,\frac{15t}{2}+4,\frac{95t}{2}+42\}$&
			$\{0,3t+1,\frac{43t}{2}+17,34t+30\}$&$\{0,3t+2,8t+7,57t+52\}$\\
			$\{0,\frac{7t}{2}+3,\frac{31t}{2}+13,\frac{97t}{2}+45\}$&$\{0,4t+2,30t+27,59t+53\}$&
			$\{0,4t+3,20t+18,36t+34\}$\\$\{0,7t+4,\frac{33t}{2}+15,58t+53\}$&
			$\{0,7t+5,36t+30,60t+53\}$&$\{0,8t+8,\frac{19t}{2}+10,\frac{141t}{2}+65\}$\\
			$\{0,8t+10,12t+14,31t+30\}$&$\{0,11t+7,22t+17,67t+60\}$&
			$\{0,\frac{23t}{2}+10,\frac{59t}{2}+25,47t+40\}$\\$\{0,15t+12,28t+26,35t+29\}$&
			$\{0,16t+14,19t+17,41t+38\}$&$\{0,17t+16,22t+20,61t+58\}$\\
			$\{0,21t+16,36t+29,48t+41\}$&$\{0,22t+19,30t+28,48t+45\}$&
			$\{0,24t+24,35t+30,39t+35\}$\\$\{0,25t+25,33t+29,33t+30\}$&
			$\{0,28t+25,36t+31,42t+37\}$&\\\hline
\end{longtable}
\end{center}

\end{document}